\documentclass[10pt]{article}

\usepackage{amssymb, latexsym, amsmath}
\usepackage{amsfonts}
\usepackage{amsthm}
\usepackage{bbm}
\usepackage[T1]{fontenc}
\usepackage{courier}
\usepackage{ae,aecompl}
\usepackage[vmargin={1.25in,1.25in},
            hmargin={1.3in,1.3in}
            ]{geometry}
\usepackage{authblk}

\theoremstyle{definition}
\newtheorem{theorem}{Theorem}[section]
\newtheorem{lemma}[theorem]{Lemma}

\newtheorem{definition}[theorem]{Definition}

\newtheorem{claim}[theorem]{Claim}

\newtheorem{question}[theorem]{Question}
\newtheorem{conjecture}[theorem]{Conjecture}

\def\EE{{\mathbb E}}
\def\PP{{\mathbb P}}
\def\NN{{\mathbb N}}
\def\ZZ{{\mathbb Z}}

\def\br{\operatorname{br}}
\def\B{\operatorname{B}}
\def\Bin{\operatorname{Bin}}

\def\({\left(}
\def\){\right)}

\newcommand{\closure}[1]{\left\langle #1 \right\rangle}

\newenvironment{proofOfTheorem}[1]{\begin{proof}[\textit{Proof of Theorem \ref{#1}}]}{\end{proof}}

\begin{document}

\title{Bootstrap percolation on Galton--Watson trees}

\author{B{\'e}la Bollob{\'a}s\thanks{Department of Pure Mathematics and Mathematical Statistics, University of Cambridge, Wilberforce Road, Cambridge CB3 0WB, UK, and Department of Mathematical Sciences, University of Memphis, Memphis, Tennessee 38152, USA, and London Institute for Mathematical Sciences, 35a South St, Mayfair, London W1K 2XF, UK. Supported in part by EU project MULTIPLEX no. 317532.}, Karen Gunderson\thanks{Heilbronn Institute for Mathematical Research, School of Mathematics, University of Bristol, Bristol BS8 1TW, UK.}, Cecilia Holmgren\thanks{Department of Mathematics, Stockholm University, 114 18 Stockholm, Sweden, and Department of Pure Mathematics and Mathematical Statistics, University of Cambridge, Wilberforce Road, Cambridge CB3 0WB, UK. Supported in part by the Swedish Research Council.}, Svante Janson\thanks{Department of Mathematics, Uppsala University, SE-75310 Uppsala, Sweden. Supported in part by the Knut and Alice Wallenberg Foundation.}, \\ and Micha{\l} Przykucki\thanks{Department of Pure Mathematics and Mathematical Statistics, University of Cambridge, Wilberforce Road, Cambridge CB3 0WB, UK.}}

\date{April 03, 2013}

\maketitle

\begin{abstract}
\small Bootstrap percolation is a type of cellular automaton which has been used to model various physical phenomena, such as ferromagnetism. For each natural number $r$, the $r$-neighbour bootstrap process is an update rule for vertices of a graph in one of two states: `infected' or `healthy'. In consecutive rounds, each healthy vertex with at least $r$ infected neighbours becomes itself infected. Percolation is said to occur if every vertex is eventually infected.

Usually, the starting set of infected vertices is chosen at random, with all vertices initially infected independently with probability $p$. In that case, given a graph $G$ and infection threshold $r$, a quantity of interest is the critical probability, $p_c(G,r)$, at which percolation becomes likely to occur. In this paper, we look at infinite trees and, answering a problem posed by Balogh, Peres and Pete, we show that for any $b \geq r$ and for any $\epsilon > 0$ there exists a tree $T$ with branching number $\br(T) = b$ and critical probability $p_c(T,r) < \epsilon$. However, this is false if we limit ourselves to the well-studied family of Galton--Watson trees. We show that for every $r \geq 2$ there exists a constant $c_r>0$ such that if $T$ is a Galton--Watson tree with branching number $\br(T) = b \geq r$ then
\[
 p_c(T,r) > \frac{c_r}{b} e^{-\frac{b}{r-1}}.
\]
We also show that this bound is sharp up to a factor of $O(b)$ by giving an explicit family of Galton--Watson trees with critical probability bounded from above by $C_r e^{-\frac{b}{r-1}}$ for some constant $C_r>0$.
\end{abstract}

\small



\normalsize

\section{Introduction and results}\label{sec:intro}

Bootstrap percolation, introduced by Chalupa, Leath and Reich
\cite{bootstrapbethe} in 1979, is one of the simplest examples of cellular
automata. Given a graph $G$ and a natural number $r\geq 2$, 
the $r$-\emph{neighbour bootstrap process} can be defined as follows. For any subset of vertices $A \subset V(G)$, set $A_0 = A$, for each $t\geq 1$ let
\begin{equation*}
 A_t = A_{t-1} \cup \{v \in V(G): |N(v) \cap A_{t-1}| \geq r\},
\end{equation*}
where $N(v)$ is the neighbourhood of $v$ in $G$. The \emph{closure} of a set
$A$ is $\closure{A} = \bigcup_{t=0}^{\infty} A_t$.  Often, this process is
thought of as the spread of an `infection' through the vertices of $G$ in
discrete time steps, with the vertices in one of two possible states:
`infected' or `healthy'.  For each $t$, $A_t$ is the set of infected
vertices at time $t$ and $\closure{A}$ is the set of vertices eventually
infected when $A$ is the set of initially infected vertices.  Given a set
$A$ of initially infected vertices, \emph{percolation} or \emph{complete
  occupation} is said to occur if $\closure{A} = V(G)$. 

Bootstrap percolation may be thought of as a monotone version of the Glauber dynamics of the Ising model of ferromagnetism. To mimic the behaviour of ferromagnetic materials, in the classical setup, all vertices of $G$ are assumed to belong to the set $A$ of initially infected vertices independently with probability $p$. It is clear that the probability of percolation is non-decreasing in $p$ and for a finite or infinite graph $G$ one can define the \emph{critical probability}
\begin{equation}
\label{eq:p_c}
p_c(G,r) = \inf \{ p : \PP_p (\closure{A} = V(G)) \geq 1/2 \},
\end{equation}
for which percolation becomes more likely to occur than not. Indeed, much work has been done in this direction for various underlying graphs and values of the infection threshold.

The question of critical probability has been studied extensively in the cases of grid-like and cube-like graphs. For example, Aizenman and Lebowitz \cite{metastabilityeffects} showed that $p_c ([n]^2,2)$ decreases logarithmically with $n$. This was later sharpened by Holroyd \cite{sharpmetastability} who showed that $p_c([n]^2,2) = \frac{\pi^2}{18 \log n} + o(1/\log n)$. Balogh, Bollob{\'a}s, Duminil-Copin and Morris \cite{sharpbootstrapall} generalized Holroyd's result giving a formula for $p_c([n]^d,r)$ for all values of $d$ and $r$. A sharp result for critical probability in $2$-neighbour bootstrap percolation on the hypercube graph was obtained by Balogh, Bollob{\'a}s and Morris \cite{bootstraphigh}.

Other types of graphs have also been studied. Janson, {\L}uczak, Turova
and Vallier \cite{bootstrapgnp} considered the random graph $G_{n,p}$,
Balogh and Pittel \cite{randomregular} worked with random regular graphs,
which were further studied by Janson \cite{givenDegrees}.
Chalupa, Leath and Reich \cite{bootstrapbethe} considered infinite
regular trees, also called Bethe lattices, which have been subsequently
examined by Balogh, Peres and Pete \cite{infiniteTrees}, by Biskup and
Schonmann \cite{metastableTrees} and by Fontes and Schonmann
\cite{homogeneousTrees}. In particular, Balogh, Peres and Pete
\cite{infiniteTrees} built upon the known results concerning bootstrap
percolation on regular trees and investigated more general results on
critical probabilities for infinite trees.  For an infinite tree $T$, the
critical probability for $r$-neighbour bootstrap percolation, denoted
$p_c(T,r)$, is defined as 
\[
p_c(T,r) = \inf\{p \mid \PP_p(T \text{ percolates in $r$-neighbour bootstrap percolation}) > 0\}.
\]
Note that this definition of $p_c(T,r)$ is different from that given in \eqref{eq:p_c}. This modification is motivated by the fact that for a general infinite tree the exact probability of percolation could be highly affected by finite, yet difficult to infect from the outside, subtrees. The existence of such substructures does not matter when we care only about the probability of percolation being positive.

For every $d \geq 1$, let $T_d$ denote the infinite $(d+1)$-regular
tree. Balogh, Peres and Pete \cite{infiniteTrees}, 
expanding the work of Chalupa, Leath and Reich \cite{bootstrapbethe}, 
gave a formula for $p_c(T_d,r)$ showing, in particular, that for any 
$d\ge1$ and $r\ge2$ 
we have $p_c(T_d,r) > 0$. They also showed that every infinite tree $T$ with
\emph{branching number} $\br(T) < r$ has the property that $p_c(T,r)=1$. 
(The branching number is defined in Section \ref{sec:small_crit_prob}.) 
Given these results, the question was raised of finding the smallest critical probability among all trees with a fixed branching number. With a simple example of a Galton--Watson tree it was shown in \cite{infiniteTrees} that for $b \geq r$ a $(b+1)$-regular tree does not, in general, minimize the critical probability for $r$-neighbour bootstrap percolation among all trees with branching number $b$. Defining a function $f_r$, for each $r \geq 2$, by
\begin{equation*}
 f_r(b) = \inf \{p_c(T, r) \mid \br(T) \leq b \text{ and $T$ has bounded degree}\},
\end{equation*}
Balogh, Peres and Pete \cite{infiniteTrees} posed the following two problems:
\begin{enumerate}
 \item Is $f_r(b)$ strictly positive for all real $b \geq 1$?
 \item Is $f_r(b)$ continuous apart from $b = r$?
\end{enumerate}
In this paper we answer both of these questions by showing that $f_r(b)$ is a step-function. More precisely, in Section \ref{sec:small_crit_prob}, we prove the following theorem.
\begin{theorem}
\label{thm:trivialF}
For all $r \geq 2$ and $b \geq r$, $f_r(b) = 0$.
\end{theorem}

Combining Theorem \ref{thm:trivialF} with the result of Balogh, Peres and
Pete \cite{infiniteTrees}, we have
\begin{equation*}
 f_r(b) = \begin{cases}
1, & \text{if }b < r,\\
0, & \text{otherwise}.
\end{cases}
\end{equation*}

We shall prove Theorem \ref{thm:trivialF} by producing trees with
arbitrarily small critical probabilities.  Motivated by the non-homogeneous
nature of these trees we also study a well--known family of well-behaved
trees: \emph{Galton--Watson trees}. For a non-negative integer--valued distribution
$\xi$, let $T_\xi$ be the Galton--Watson tree with offspring distribution
$\xi$ (a more formal definition is given in Section \ref{sec:GWtrees}). 
 We shall see in Section \ref{sec:GWtrees} that $p_c(T_\xi,r)$ is almost
surely a constant (depending on the distribution $\xi$ but not on the
realization $T_\xi$); we let $p_c(T_\xi,r)$ denote also this constant,
without risk of confusion.
We define a new function $f_r^{GW}\!(b)$ by
\begin{equation}
 \label{eq:frGWb}
 f_r^{GW}\!(b) = \inf \{p_c(T_\xi,r) \mid \EE(\xi)=b, \PP(\xi = 0) = 0\}.
\end{equation}
The condition that $\PP(\xi = 0) = 0$ is included since any finite tree percolates with positive probability if the probability of initial infection, $p$, is positive.  For this reason, we consider only offspring distributions for which the resulting tree is almost surely infinite.  While the branching numbers of infinite trees can be difficult to determine, for Galton--Watson trees, Lyons \cite{walksOnTrees} showed that, almost surely, $\br(T_{\xi}) = \EE(\xi)$. 

In Section \ref{sec:GWtrees}, we shall investigate the function $f_r^{GW}\!(b)$ and we shall show it to be positive for all $b$ and $r$.  That is, the value of $\EE(\xi)$ immediately leads to a non-trivial lower bound on $p_c(T_\xi, r)$. We shall also show that our bound is tight up to a factor of $O(b)$. 
\begin{theorem}
\label{thm:branchingBound}
Let the function $f_r^{GW}\!(b)$ be defined as in \eqref{eq:frGWb}.
\begin{enumerate}
 \item If $r > b \geq 1$ then $f_r^{GW}\!(b) = 1$.
 \item For $r \geq 2$ there are constants $c_r$ and $C_r$ such that if $b \geq r$ then
       \begin{equation*}
        \frac{c_r}{b} e^{-\frac{b}{r-1}} \le f_r^{GW}\!(b) \le C_r e^{-\frac{b}{r-1}}.
       \end{equation*}
\end{enumerate}
\end{theorem}

Note that the $b$-ary tree is a Galton--Watson tree given by $\xi$ with $\PP(\xi = b)=1$.  The $b$-ary tree has the same critical probability as the $(b+1)$-regular tree $T_b$.  By Theorem \ref{thm:branchingBound}, for large $b$, the value of $f_r^{GW}\!(b)$ is extremely far from the value $p_c(T_b, b) = 1-\frac{1}{b}$, obtained in \cite{infiniteTrees}. This discrepancy suggests that offspring distributions highly concentrated around their means might yield much higher values for the critical probability. This is in fact true as shown by the following theorem, proved in Section \ref{sec:1+alpha}.
\begin{theorem}
\label{thm:(1+alpha)Bound}
For each $r \geq 2$ and $\alpha \in (0,1]$ there exists a constant
$c_{r,\alpha}>0$   such that for any offspring distribution $\xi$ we have 
\[
 p_c(T_\xi,r) \ge  c_{r,\alpha} \left( \EE(\xi^{1+\alpha}) \right)^{-1/\alpha} .
\]
Also, for each $r \geq 2$ there exists a constant $A_r > 0$ such that
\[
 p_c(T_\xi,r) \leq  \EE \left( \frac{A_r}{\xi^{r/(r-1)}} \right).
\]
\end{theorem}

The lower bound in Theorem \ref{thm:(1+alpha)Bound} is proved directly for
$\alpha \in (0,1)$. For $r \geq 3$ the constants $c_{r, \alpha}$ obtained in the theorem converge to 
$c_r > 0$ as $\alpha \to 1$ and hence by continuity, the theorem holds 
for $r \geq 3$ and $\alpha = 1$. For $r = 2$ and $\alpha = 1$ the theorem
holds by the final result in this paper, given in Section
\ref{sec:lb_2nbr}. There, we prove the following theorem which, apart from a
sharp lower bound on $p_c(T_{\xi}, 2)$ based on the second moment of $\xi$,
 also gives additional lower bounds on the critical probability in
$2$-neighbour bootstrap percolation, as well as a sharp upper bound on
$p_c(T_{\xi}, 2)$ based on the second negative moment of $\xi$. 
\begin{theorem}
\label{thm:2ndMomentBound}
Let $T_\xi$ be the Galton--Watson tree of an offspring distribution $\xi$. Then
\begin{equation}
\label{eq:dominationBound}
 p_c(T_{\xi}, 2) \geq \max \left \{1 - \frac{1}{2 \PP(\xi=2)}, \max_{k \geq
   3} \left \{ 1 - \frac{(k-1)^{2k-3}}{k^{k-1}(k-2)^{k-2} \PP(\xi=k)} \right
 \} \right \},
\end{equation}
and 
\begin{equation}\label{eq:dominationUpperbd}
 p_c(T_{\xi}, 2) \leq \EE\left(\frac{1}{(\xi-1)(2\xi-3)}\right)
\le \EE\left(\frac {4}{\xi^2} \right).
\end{equation}

Additionally, if $\xi$ has the property that $\EE(\xi^2) < \infty$, then
\begin{equation}
\label{eq:2ndMomentBound}
 p_c(T_{\xi}, 2) \geq \frac{1}{2 \EE(\xi(\xi-1)) - 3}
\ge \frac 1{2\EE \xi^2}.
\end{equation}
\end{theorem}

Balogh, Peres and Pete \cite{infiniteTrees} noted that as $b \to \infty$, the critical probability for the regular tree, $T_b$, is $p_c(T_b, 2) \sim \frac{1}{2b^2}$, which matches the bounds given in Theorem \ref{thm:2ndMomentBound}.

Finally, in Section \ref{sec:examples} we shall present some examples of natural classes of Galton--Watson trees for which the critical probability for bootstrap percolation can be computed exactly and compare these to the bounds given by Theorem \ref{thm:2ndMomentBound}. To conclude, in Section \ref{sec:openPr}, we state a few questions and conjectures.

\section{Trees with arbitrarily small critical probability}\label{sec:small_crit_prob}

In this section, a construction is given for families of infinite trees with
a fixed branching number and arbitrarily small critical probability.

The branching number is one of the most important invariants of infinite
trees which we shall now define formally. 
(For further information, see, for example, Lyons \cite{walksOnTrees}.)
Given a rooted tree $T$, for every edge $e$ in the tree, let $|e|$ denote the number of edges (including $e$) in the path from $e$ to the root.  The branching number of a tree $T$, denoted $\br(T)$, is the supremum of real numbers $\lambda \geq 1$ such that there exists a positive flow in $T$ from the root to infinity with capacities at every edge $e$ bounded by $\lambda^{-|e|}$. It is easily seen that this value does not depend on the choice of the root. Though in this paper, only infinite trees are considered,  let us mention that for a finite tree $T$ we have $\br(T) = 0$.

For $b\geq 2$, let $T_b$ denote the infinite $(b+1)$-regular tree. As usual,
for $n \geq 1$ and $p \in [0,1]$, write $\Bin(n,p)$ for a binomial random
variable with parameters $n$ and $p$. In 
\cite{bootstrapbethe}, it was shown that, in $r$-neighbour bootstrap percolation, for each $b \geq r$, the critical probability $p_c(T_b, r)$ 
 is equal to the supremum of
all $p$ for which the fixed-point equation 
\begin{equation}
\label{E:fixedpt}
 x = \PP(\Bin(b, (1-x)(1-p)) \leq b-r)
\end{equation}
has a solution $x \in [0,1)$. Note that $x=1$ is always a solution to equation \eqref{E:fixedpt}.

An interpretation of equation \eqref{E:fixedpt} is as follows. The complete occupation of $T_b$ obeys the $0-1$ law and can be shown to be stochastically equivalent to complete occupation of a rooted $b$-ary tree, that is, a rooted infinite tree in which every vertex has exactly $b$ descendants (so all vertices have degree $b+1$ except the root which has degree $b$). For $b \geq r$ the root of a $b$-ary tree, conditioned on being initially healthy, remains healthy forever if{f} at least $b-r+1$ of its children are initially healthy and remain healthy forever. Let $x$ be the probability that, conditioned on being initially healthy, the root \emph{does not} remain healthy forever. Then, one can show that $x$ is the smallest solution to equation \eqref{E:fixedpt} in $[0,1]$.  In particular, it was noted in \cite{bootstrapbethe} that $p_c(T_b, 2) = 1- \frac{(b-1)^{2b-3}}{b^{b-1}(b-2)^{b-2}}$ and later in \cite{infiniteTrees} that $p_c(T_b, b) = 1-\frac{1}{b}$.  It can be shown that for every fixed $r$, as $b$ tends to infinity, $p_c(T_b, r) = \left(1 - \frac{1}{r}\right) \left(\frac{(r-1)!}{b^r} \right)^{1/(r-1)}(1+o(1))$.  This calculation is given in Lemma \ref{lem:pc_regtree}, to come.

From equation \eqref{E:fixedpt} we see immediately that $p_c(T_b, r) > 0$ for any $b \geq r \geq 2$. In \cite{infiniteTrees} the authors asked whether there exists $\epsilon_{b,r} > 0$ such that for any tree $T$ with branching number $\br(T) = b$ we have $p_c(T, r) \geq \epsilon_{b,r}$, answering this question affirmatively for $r>b$ with $\epsilon_{b,r} = 1$.

With an explicit construction of a family of infinite trees with bounded degree we shall now show that $f_r(b) = 0$ for $b \geq r$. The condition that the tree $T$ has bounded degree is included in the definition of the function $f_r(b)$ since one can easily construct infinite trees with unbounded degree and branching number $b$, and such that their critical probability is $0$. We show an example of such construction at the end of this section.

Given $r \geq 2$, $b \geq r$ and $p \in (0,1)$, we shall show that there is an integer $d$ and an infinite tree with branching number $b$ where every vertex has either degree $d+1, d+2, b+1$ or $b+2$ and such that, infecting vertices with probability $p$, the tree almost surely percolates. The rough idea of the proof is that, when $d$ is sufficiently large, vertices that are the roots of some finite number of levels of a copy of $T_d$ are very likely to eventually become infected  and these finite trees can be arranged within an infinite tree to cause the percolation of the entire tree.

First, it is shown that, for the infection threshold $r$ and for $d$ large enough, we can in fact obtain an arbitrarily small critical probability $p_c(T_d, r)$.
\begin{lemma}
\label{lemma:largeD}
For each integer $r \geq 2$ and $d \geq r$,  $p_c(T_d, r) \leq r/d$.
\end{lemma}
\begin{proof}
Fix $r \geq 2$, $d \geq r$ and $p \geq r/d$.  To prove this result, it suffices to show that for all $x \in [0,1)$ we have
\[
 \PP(\Bin(d, (1-x)(1-p)) \leq d-r) > x,
\]
or alternatively,
\[
 \PP(\Bin(d, (1-x)(1-p)) \geq d-r+1) < 1-x.
\]
In this case there are no solutions of the fixed point equation \eqref{E:fixedpt} in $[0,1)$ and so $p_c(T_d, r) \leq p$.

Recall the following Chernoff-type inequality: if $X \sim \Bin(n,p)$ and $m \geq np$, then $\PP(X \geq m) \leq e^{-np} (enp/m)^m$. Since $dp \geq r$,
\begin{align*}
\PP(\Bin&(d, (1-x)(1-p)) \geq d-r+1)\\
	&\leq e^{d-r+1-d(1-x)(1-p)} \left( \frac{ d (1-x)(1-p)}{d-r+1} \right)^{d-r+1}\\
	&= e^{d-r+1-d(1-x)(1-p)}\left(\frac{d(1-p)}{d-r+1}\right)^{d-r+1} (1-x)^{d-r} (1-x)\\
	&\leq  e^{d-r+1-d(1-x)(1-p)} \left(1-\frac{dp-r+1}{d-r+1}\right)^{d-r+1}e^{-x(d-r)} (1-x)\\
	&\leq \exp\left[d-r+1-d(1-x)(1-p) - (dp-r+1) -x(d-r)\right] (1-x)\\
	&= \exp(-x(dp-r))(1-x)\\
	&< 1-x,
\end{align*}
for all $x \in [0,1)$. Thus, there are no solutions of equation \eqref{E:fixedpt} in $[0,1)$ and hence $p_c(T_d, r) \leq p$.
\end{proof}

As a consequence of Lemma \ref{lemma:largeD}, for $r$ fixed, $\lim_{d \to \infty} p_c(T_d,r) =0$.

In the next lemma we show that, for any $\epsilon \in (0,1)$, there is a large number $n_{\epsilon}$ such that if we initially infect vertices in the first  $n_\epsilon$ levels of $T_d$ with probability $p \geq p_c(T_d,r)$, then the root of $T_d$ will become infected in the $r$-neighbour bootstrap process with probability at least $1-\epsilon$.  For any $d \geq 1$, $n \geq 0$, let $T_d^n$ be the first $n+1$ levels of a rooted, $(d+1)$-regular tree.  That is, the root has $d+1$ children, there are $(d+1)d^{n-1}$ leaves and every vertex except the root and the leaves has exactly $d$ children.
\begin{lemma}
\label{lemma:levelsOfT_d}
For $d \geq r \geq 2$, $p > p_c(T_d,r)$, and $n \geq 1$, let the vertices
of $T_d^{n}$  be infected independently with probability $p > 0$.  For
the $r$-neighbour bootstrap process, 
\[
 \PP_p(\text{the root of } T_d^{n} \text{ is eventually infected}) \to 1
\]
as $n \to \infty$.
\end{lemma}

\begin{proof}
Note that if $p > p_c(T_d,r)$ then for $r$-neighbour bootstrap percolation on $T_d$, using a $0-1$ law argument, $\PP_p(T_d \text{ percolates}) = 1$ and hence 
\[
\PP_p(\text{root is eventually infected}) = \PP_p(\cup_{t \geq 0} \{\text{root is infected by time } t\}) = 1.
\]
Using induction, one can show that the root is infected by time $t$ exactly when the eventual infection of the root depends on the infection status of vertices in the first $t$ levels.  Indeed, if the root is infected at time $0$, this event depends only on the initial infection of the root itself.  For $t \geq 1$, if the root becomes infected at time $t$, then at least $r$ of its children are infected at time $t-1$.  By induction this event depends only on vertices at distance at most $t-1$ from the children of the root and hence at distance at most $t$ from the root itself.

Therefore, $\lim\limits_{t \to \infty} \PP_p(\text{root infected based on first $t$ levels}) = 1$. 
\end{proof}

We are now ready to prove Theorem \ref{thm:trivialF} with the construction
given in the proof of Theorem \ref{thm:small_crit_prob} below.

\begin{theorem}
 \label{thm:small_crit_prob}
 For every pair of integers $r \geq 2$ and $b \geq r$ and every $p \in (0,1)$, there is an infinite tree $T$ with bounded degree and $\br(T)=b$ satisfying $p_c(T,r)<p$.
\end{theorem}
\begin{proof}
Fix $p \in (0,1)$ and integers $r, b$ with $b \geq r$.  
Let $d > \max\{r/p, b\}$ 
so that, by Lemma \ref{lemma:largeD}, $p > r/d \geq p_c(T_d, r)$. Let
$\{n_i\}_i$ and $\{m_i\}_i$ be sequences of integers, all to be defined
precisely later in the proof.  Our tree is constructed level-by-level,
depending on these parameters; it will be shown that the sequences
$\{n_i\}_i$ and $\{m_i\}_i$ can be chosen appropriately so that the
resulting tree has the desired properties. 

Begin with a copy of $T_d^{n_1}$. To each leaf of this compound tree attach
a copy of $T_b^{m_1}$.  Then to each leaf of the resulting tree attach a copy
of $T_d^{n_2}$ and then to each new leaf attach a copy of $T_b^{m_2}$.  Continue
in this manner, alternating with $(d+1)$-regular trees and $(b+1)$-regular
trees of depths given by the sequences $\{n_i\}_i$ and $\{m_i\}_i$
respectively and let $T$ be the resulting infinite tree. We would like to
show that there is a suitable choice for  the sequences $\{n_i\}$ and
$\{m_i\}$ so that $\br(T) = b$ and $p_c(T,r) < p$ (in other words, $\PP_p(T
\text{ percolates}) >0$). 

For each $\ell \geq 1$, let $N_{\ell}=\prod_{i=1}^{\ell-1} (d+1)d^{n_i-1} (b+1)b^{m_i-1}$ be the number of copies of $T_d^{n_{\ell}}$ added in the $(2\ell-1)$-th step of the construction and let $v_1^{\ell}, v_2^{\ell}, \ldots, v_{N_{\ell}}^{\ell}$ be the roots of those copies of $T_d^{n_{\ell}}$ and let $T^{n_\ell}_{d, i}$ denote the copy of $T^{n_\ell}_d$ rooted at $v^\ell_i$.  Define $t_{\ell} = \sum_{i=1}^{\ell-1}(n_i + m_i)$ to be the depth of these vertices in $T$. For each $\ell \geq 1$ and  $i \in \{1, \ldots, N_{\ell}\}$, consider the event 
\[
A_{\ell, i} = \{v_i^{\ell} \text{ becomes infected based only on infection of vertices in $T_{d, i}^{n_{\ell}}$}\}.
\]
Using Lemma \ref{lemma:levelsOfT_d}, choose $n_{\ell}$ to be large enough so that $\PP(A_{\ell,i}) \geq (1/2)^{1/N_{\ell}}$. Note that $N_{\ell}$ does not depend on $n_{\ell}$.  Set $A_{\ell} = \cap_i A_{\ell,i}$.  If $A_{\ell}$ occurs, then all vertices in level $t_{\ell}$ are eventually infected and hence all vertices in levels at most $t_{\ell}$ are eventually infected. Further, if infinitely many events $\{A_{\ell}\}_{\ell}$ occur, then $T$ percolates.

For $\ell$ fixed, since the events $\{A_{\ell, i}\}_i$ are independent, by the choice of $n_{\ell}$ we have
\[
\PP(A_{\ell}) = \PP(\cap_i A_{\ell, i}) = \prod_{i=1}^{N_{\ell}} \PP(A_{\ell, i}) \geq \prod_{i=1}^{N_{\ell}} \left(\frac{1}{2} \right)^{1/N_{\ell}} = \frac{1}{2}.
\]
By the Borel-Cantelli lemma, since the events $\{A_{\ell}\}$ are independent and
\[
 \sum_{\ell} \PP(A_{\ell}) \geq \sum_{\ell} \frac{1}{2}  = \infty,
\]
then $\PP(T \text{ percolates}) = 1$.

Up to this point, no conditions have been imposed on the sequence $\{m_i\}_i$ and these can be chosen, in such a way that $\br(T) = b$.  Note that, since $d$ was chosen with $d > b$, every vertex of $T$ has at least $b$ children and so $\br(T) \geq b$.  By a choosing the values of $m_i$ recursively, depending on the sequence $\{n_i\}$, it is shown below that $\br(T) \leq b$.

For every $n$, let $L_n$ be the $n$-th level of $T$, i.e., the vertices at distance $n$ from the root of $T$. A standard upper bound on the branching number of an arbitrary tree gives $\br(T) \leq \liminf |L_n|^{1/n}$.

For $\ell \geq 1$, consider the level $t_{\ell+1} = \sum_{i=1}^{\ell} (n_i+m_i)$ with $\prod_{i=1}^{\ell} (d+1)d^{n_i-1}(b+1)b^{m_i-1}$ vertices. Clearly, if $m_\ell \geq \ell^2$ is large enough then
\[
\left(\frac{d}{b} \right)^{\frac{\sum_{i=1}^{\ell} n_i}{t_{\ell+1}}} \leq 1 + \frac{1}{2^{\ell}}
\]
and $\ell/t_{\ell+1} \to 0$ as $\ell \to \infty$. Then, the number of vertices in level $t_{\ell+1}$ satisfies
\begin{align*}
|L_{t_{\ell+1}}|
	&=\prod_{i=1}^{\ell} (d+1)d^{n_i-1} (b+1)b^{m_i-1}\\
	&= b^{t_{\ell+1}} \left(\frac{d}{b}\right)^{\sum_{i=1}^{\ell} n_i}\left(1+\frac{1}{d}\right)^{\ell}\left(1+\frac{1}{b}\right)^{\ell}\\
	&\leq b^{t_{\ell+1}} \left(1 + \frac{1}{2^{\ell}} \right)^{t_{\ell+1}}\left(1+\frac{1}{d}\right)^{\ell}\left(1+\frac{1}{b}\right)^{\ell}.
\end{align*}

Thus, $\liminf |L_n|^{1/n} \leq b$ and so $\br(T) = b$. 
\end{proof}

For simplicity, the proof of Theorem \ref{thm:small_crit_prob} assumes that $b$ is an integer.  For any real $b \geq r$, the construction can be modified to give an infinite tree with branching number $b$ and arbitrarily small critical probability. 

By Theorem \ref{thm:small_crit_prob}, for $b \geq r$, $f_r(b) =0$, completing the proof of Theorem \ref{thm:trivialF}.

The construction in the proof of Theorem \ref{thm:small_crit_prob} can also be modified to produce examples of infinite trees with branching number $b$, unbounded degree and critical probability $0$. Indeed, set $n_i \equiv 1$, and for each $\ell \geq 1$, at step $2\ell-1$ of the construction replace $d$ by $d_\ell$, chosen to be large enough so that for the corresponding events $A_{\ell, i}$,
\[
 \PP(A_{\ell,i})  = \PP (\Bin(d_\ell+1,1/\ell) \geq r) \geq \left(\frac{1}{2} \right)^{1/N_{\ell}}.
\]
The sequence $\{m_i\}_i$, giving the number of levels of the $(b+1)$-regular trees, can be chosen to ensure $\br(T) = b$. The resulting infinite tree $T$ has branching number $b$, unbounded degree and $p_c(T,r) = 0$.

\section{Critical probabilities for Galton--Watson trees}
\label{sec:GWtrees}

\subsection{Definitions}\label{subs:GW_defs}

In the previous section, we showed that the branching number $\br(T)$ of an
infinite tree $T$ does not lead to any nontrivial lower bound on the
critical probability $p_c(T, r)$, except when $\br(T) < r$ and $p_c(T,
r)=1$, as shown in \cite{infiniteTrees}. The trees constructed in the proof
of Theorem \ref{thm:small_crit_prob} to show that if $b \geq r$, then $f_r(b)=0$, are
highly non-homogeneous and the irregularities in their construction seem
crucial to their small critical probabilities. In this section we limit our
attention to the well--studied family of Galton--Watson trees, for which
these anomalies do not occur.

A Galton--Watson tree is the family tree of a Galton--Watson branching process. For a non-negative integer-valued distribution $\xi$, called the \emph{offspring distribution}, we start with a single root vertex in level $0$ and at each generation $n = 1, 2, 3, \ldots$ each vertex in level $n-1$ gives birth to a random number of children in level $n$, where the number of offspring of each vertex is distributed according to the distribution $\xi$ and independent of the number of children of any other vertex. This process can be formalized to define a probability measure on the space of finite and infinite rooted trees and $T_\xi$ is used to denote a randomly chosen Galton--Watson tree with offspring distribution $\xi$. As previously mentioned, if $\PP(\xi = 0) > 0$ then $T_\xi$ is  finite with positive probability. Thus in this paper we limit our attention to offspring distributions with $\PP(\xi = 0) = 0$ for which $T_\xi$ is almost surely infinite.

While the critical probability $p_c(T_\xi, r)$ is a random variable, which could take a range of values, depending on the tree $T_\xi$, it can be shown that in the space of Galton--Watson trees with offspring distribution $\xi$, conditioned on $T_\xi$ being infinite, $p_c(T_\xi, r)$ is almost surely a constant.  While this involves standard applications of results and techniques in the theory of branching processes, the details are given in this section for completeness.

For any rooted tree $T$, with root $v_0$, let $\{T_w \mid w \in N(v_0)\}$ be the collection of rooted sub-trees of $T$ whose roots are the immediate descendants of $v_0$; that is, $T_w$ is the connected component of $T - v_0$ containing $w$ and rooted at $w$.
A property ${\cal A}$ of rooted trees is called \emph{inherited} if every
finite tree $T$ has this property and, furthermore, if $T$ has the property
${\cal A}$ if and only  if for every $w$ adjacent to the root, $T_w$ has property $\cal{A}$ also. It can be shown that for a
Galton--Watson tree, conditioned on the survival of the process,  every
inherited property has conditional probability either $0$ or $1$ (see, for
example, Proposition 5.6 in \cite{treesNetworks}). 

Given $p > 0$ and $r \geq 2$ consider the property
\[
 {\cal A}_p =  \{ \PP_p(T \text{ percolates in the } r \text{-neighbour bootstrap process}) > 0 \}.
\]
Clearly, the property ${\cal A}_p$ is inherited.  Since we consider offspring distributions with $\PP(\xi = 0) = 0$,  the Galton--Watson process survives almost surely and we see that the probability that the Galton--Watson tree $T_\xi$ has property ${\cal A}_p$ is either $0$ or $1$. By the definition of critical probability this implies that $p_c(T_\xi, r)$ is almost surely a constant. 

Before proving Theorem \ref{thm:branchingBound}, let us recall the following
definition from \cite{infiniteTrees}.
\begin{definition}\label{def:r-fort}
Let $G$ be a graph and $r \in \mathbb{Z}^+$.  A
finite or infinite set of vertices, $F \subset V(G)$, is called an \emph{$r$-fort} if{f} every vertex in $F$ has
at most $r$ neighbours in $V(G) \setminus F$.
\end{definition}

 While a fort is a subgraph of the graph $G$, not depending on the infection status of vertices, if $G$ contains an $(r-1)$-fort, $F$, with all vertices initially healthy, then $G$ does not percolate in
 the $r$-neighbour bootstrap process.  
 Moreover, the set of eternally healthy vertices is an $(r-1)$-fort, so 
a vertex remains healthy forever if and only if it belongs to a healthy
$(r-1)$-fort.

Now we show that we may assume that $\PP(\xi < r) = 0$, repeating the argument observed earlier in \cite{infiniteTrees}. If there is a $k <r$ such that $\PP(\xi = k) > 0$, then $T_\xi$ almost surely contains infinitely many pairs of vertices $u,v$ such that $v$ is a child of $u$ and $\deg(u)=\deg(v)=k+1$. Then, if we initially infect vertices of $T_\xi$ independently with some probability $p<1$, almost surely we obtain such a pair with both $u$ and $v$ initially healthy, in which case $\{u,v\}$ is an initially healthy $(r-1)$-fort.  Thus $T_\xi$ almost surely does not percolate and so $p_c(T_\xi, r) = 1$.

Therefore assume that $\PP(\xi < r) = 0$;
 in particular, $\EE(\xi)=b \geq r$. 
In this case, almost surely, $T_\xi$ contains no finite $(r-1)$-forts. 

In \cite{infiniteTrees}, Balogh, Peres and Pete, characterize the critical probability for a particular Galton--Watson tree in terms of the probability that the root of the tree remains healthy in the bootstrap process.  The details are given here for arbitrary Galton--Watson trees.

For any tree $T$ with root $v_0$, $r \geq 2$ and $p \geq 0$, initially infecting vertices with probability $p$, define
\[
 q(T,p) = \PP_p(v_0 \text{ is in a healthy } (r-1)\text{-fort}),
\]
the probability that $v_0$ is never infected.
Since, in general, the random variable $q(T_\xi, p)$ depends on the tree $T_\xi$, consider its expected value, over the space of random Galton--Watson trees with offspring distribution $\xi$ and set
\[
 q(p) = \EE_{T_{\xi}}(q(T_{\xi},p)).
\]

In what follows, it is shown that $q(p) > 0$ if{f} $p< p_c(T_\xi,r)$.

For a fixed tree $T$ with root $v_0$, denote the children of the root by $v_1, v_2, \ldots, v_k$ and the corresponding sub-trees by $T_1, T_2, \ldots, T_k$.  The root $v_0$ is contained in an infinite healthy $(r-1)$-fort if{f} $v_0$ is initially healthy and at least $k-r+1$ of its children are themselves contained in an infinite healthy $(r-1)$-fort in their sub-tree $T_i$. Since these $k$ events are mutually independent,
\[
 q(T,p) = (1-p) \sum_{\underset{|X| \leq r-1}{X \subseteq [1,k]}} \left(\prod_{i \in X}(1-q(T_i, p))\prod_{j \notin X} q(T_j,p)\right).
\]
If $T$ is a Galton--Watson tree with offspring distribution $\xi$ then,
given that the root has exactly $k$ children, the sub-trees $T_1, T_2, \ldots,
T_k$ are also such (independent) subtrees. Thus, 
\begin{align}
 q(p)	&=(1-p)\sum_{k \geq r} \PP(\xi =k) \sum_{i \leq r-1} \binom{k}{i}(1-q(p))^i q(p)^{k-i} \notag\\ 
	&=(1-p) \sum_{k \geq r} \PP(\xi = k) \PP(\Bin(k, 1-q(p)) \leq
 r-1). \label{eq:fixedPoint} 
\end{align} 

Define a function $h_{r,p}(x)$, depending implicitly on the distribution
$\xi$, by 
\[
 h_{r,p}(x) = (1-p) \sum_{k \geq r} \PP(\xi = k) \PP(\Bin(k, 1-x) \leq r-1).
\]
 By equation \eqref{eq:fixedPoint}, $q(p)$ is a fixed point of $h_{r,p}(x)$.
Note that this is closely related to the fixed point equation 
\eqref{E:fixedpt} from \cite{bootstrapbethe} 
with $x$ in place of $(1-p)(1-x)$.

The function $h_{r,p}(x)$ is continuous on $[0,1]$, $0 \leq h_{r,p}(x) \leq (1-p)$ and since
\begin{equation}\label{eq:diff_binom_prob}
 \frac{d}{dx} \PP(\Bin(k, 1-x) \leq r-1) = k \PP(\Bin(k-1, 1-x) = r-1) > 0
\end{equation}
for all $k \geq r$ and $0 < x < 1$, $h_{r,p}$ is strictly increasing in
$[0,1]$ unless $p=1$.  Note that for any $p$, $h_{r,p}(0) = 0$ and so $0$ is
a fixed point of the function.  Using standard techniques for branching processes, it is shown that the critical probability $p_c(T_\xi,r)$ is given as follows in terms of the function $h_{r,p}(x)$.
\begin{lemma}\label{lem:pc_hrp}
The critical probability $p_c(T_\xi,r)$ is given by
\[
p_c(T_\xi,r) = \inf\{p \mid x = h_{r,p}(x) \text{ has no solution for } x \in (0,1]\}.
\]
\end{lemma}
The proof of Lemma \ref{lem:pc_hrp} is given by Claim \ref{claim:fixedpt} and Lemma \ref{lem:q(p)} below.

\begin{claim}\label{claim:fixedpt}
 For every $p$, $q(p)$ is the largest fixed point of $h_{r,p}(x)$ in $[0,1]$.
\end{claim}

\begin{proof}
 If $p=1$ then $h_{r,p}(x) = 0$ for all $x \in [0,1]$ and so $x=0$ is the only fixed point of $h_{r,p}(x)$ in $[0,1]$. Thus $q(p)$, itself being such a fixed point, must be equal to $0$.
 
 Therefore assume that $p < 1$. For any tree $T$, let $T^{n}$ be the first $n$ levels of $T$ and define
\[
 q_n(T,p) = \PP_p(v_0 \text{ is in a healthy } (r-1)\text{-fort of } T^{n})
\]
and $q_n(p) = \EE_{T_\xi}(q_n(T_\xi, p))$.

Since the definition of a fort depends only on the neighbourhood of each
vertex, a sub-tree $F \subseteq T$ is an $(r-1)$-fort if{f} for every $n
\geq 0$, $F \cap T^{n}$ is an $(r-1)$-fort in $T^{n}$; furthermore, the latter event is decreasing in $n$.
Therefore, $q_n(T, p) \searrow q(T,p)$ as $n \to \infty$ and so also $q_n(p) \searrow q(p)$.

Following the same recursive argument as before, we see that for every $n \geq 0$, $q_{n+1}(p) = h_{r,p}(q_n(p))$.  Note also that for any tree $T$,
\[
 q_0(T,p) = \PP_p(v_0 \text{ is initially healthy}) = 1-p.
\]
Suppose that $x_0$ is a fixed point of $h_{r,p}(x)$.  Then, $x_0 = h_{r,p}(x_0) \leq 1-p = q_0(p)$.
Proceeding by induction, suppose that for some $n \geq 0$,  $x_0 \leq q_n(p)$.  Since $h_{r,p}(x)$ is increasing,
\[
 x_0 = h_{r,p}(x_0) \leq h_{r,p}(q_n(p)) = q_{n+1}(p).
\]
Therefore, $x_0 \leq \lim_{n \to \infty} q_n(p) = q(p)$, completing the proof.
\end{proof}

There is a small difference between the event that the root of a tree $T$ is the root of a healthy $(r-1)$-fort and the event that some other vertex of $T$ is the root of a healthy $(r-1)$-fort.  Fix a vertex $v$ in $T$ that is not the root and consider the probability that $v$ is the root of a healthy fort, in $T$.  Since $v$ already has a neighbour (its parent) not in the fort, then $v$ is the root of a healthy $(r-1)$-fort if{f} $v$ has at most $r-2$ children that are not, themselves, roots of healthy $(r-1)$-forts.  Thus, for $T = T_\xi$ and conditioned on $v$ being a vertex of the tree,
\begin{equation} \label{eq:root_healthy_tree}
\begin{split}
 \EE_{T_\xi}(\PP_p(v & \text{ is the root of a healthy } (r-1)\text{-fort})\mid v \in T_\xi) \\
    &=(1-p) \sum_{k \geq r} \PP(\xi=k) \PP(\Bin(k, 1-q(p))\leq r-2)\\
    &=h_{r-1, p}(q(p)).
 \end{split}
\end{equation}
Since for all $s \geq 1$ and $p < 1$ we have $h_{s,p}(x)=0$ if{f} $x=0$ then in particular, $q(p) = 0$ if{f} $h_{r-1,p}(q(p))=0$.
\begin{lemma}\label{lem:q(p)}
In the space of Galton--Watson trees for a fixed distribution $\xi$,  if $q(p)
>0$, then $\PP_p(T_\xi \text{ percolates})=0$ almost surely.  If $q(p)=0$, then $\PP_p(T_\xi \text{ percolates})=1$ almost surely.
\end{lemma}

\begin{proof}
 If $p=1$ then $q(p)=0$ and clearly $\PP_p(T \text{ percolates})=1$. So assume that $p < 1$.

First, assume that $q(p)>0$, with the aim of showing that
\[
 \EE_{T_\xi}(\PP_p(T_\xi \text{ percolates}))=0.
\]

By equation \eqref{eq:root_healthy_tree}, there is a $\delta >0$ be such that, for every vertex $v$, 
\[
\EE_{T_\xi}(\PP_p(v \text{ is in a healthy } (r-1)\text{-fort}\mid v \in T_\xi)) \geq \delta.
\]
  Since $\xi \geq r$ almost surely, at level $t$ in the tree, there are
at least $r^t$ vertices. 
The events that these vertices are roots of healthy $(r-1)$-forts are independent;
thus, for every $t$
\[
 \EE_{T_\xi}(\PP_p(\text{every vertex of } T_\xi \text{ at level } t \text{ is eventually infected})) \leq (1-\delta)^{r^t} \to 0
\]
as $t \to \infty$.  Thus, $\EE_{T_\xi}(\PP_p(T_\xi \text{ percolates})) = 0$ and hence the set
\[ 
\{T \mid \PP_p(T \text{ percolates})>0\}
\]
 has measure $0$.

On the other hand, suppose that $\EE_{T_\xi}(\PP_p(T_\xi \text{ percolates}))<1$ in hopes of showing that $q(p) >0$.  Then, the set of trees
\[
 \{T \mid \PP_p(T \text{ percolates})<1\} = \{T \mid \PP_p(T \text{ contains a healthy } (r-1) \text{-fort}) >0\}
\]
has positive measure.

Even though the number of infinite trees is uncountable, each tree has
only a countable number of vertices 
 and these can be thought of as a subset of 
a common countable set of vertices.  
Then, there is a vertex $v$ for which, conditioning on $v$ being a vertex of the tree,
\[
 \EE_{T_\xi}(\PP_p(v \text{ is the root of a healthy } (r-1) \text{-fort})\mid v \in V(T_\xi)) > 0.
\]
 That is, either $q(p) >0$ (if $v = v_0$) or $h_{r-1, p}(q(p))>0$.  In either case, $q(p)>0$, which completes the proof.
\end{proof}

Thus, combining Claim \ref{claim:fixedpt} and Lemma \ref{lem:q(p)}, Lemma \ref{lem:pc_hrp} holds and the critical probability is given by
\begin{equation}\label{eq:crit_prob}
p_c(T_\xi,r) = \inf\{p \mid x = h_{r,p}(x) \text{ has no solution } x \in (0,1]\}.
\end{equation}
With equation \eqref{eq:crit_prob} in mind, the following functions are defined.

\begin{definition}\label{def:G_xi}
For each $r \geq 2$ and $k \geq r$, define
\[
g_k^r(x) = \frac{\PP(\Bin(k, 1-x) \leq r-1)}{x} = \sum_{i = 0}^{r-1} \binom{k}{i} x^{k-i-1}(1-x)^i
\]
and for any offspring distribution $\xi$, set 
\[
G_{\xi}^r(x) = \sum_{k \geq r} \PP(\xi=k) g_k^r(x).
\]
\end{definition}

Using equation \eqref{eq:crit_prob}, the critical probability for $T_\xi$ can be characterized in terms of the function $G_\xi^r(x)$.  Note that for $p=0$, the equation $h_{r,p}(x)=x$ has a solution at $x=1$ and for $p=1$, the only solution to $h_{r,p}(x) = x$ is $x=0$.  Since $h_{r,p}(x) = x (1-p) G_{\xi}^r(x)$, then for $p<1$, $x
= h_{r,p}(x)$ has a solution in $(0,1]$ if{f} $G_{\xi}^r(x) = \frac{1}{1-p}$ has a solution in $(0,1]$. Note that we have $G^r_\xi(1)=1$, and so for $p>0$,
$(1-p)G^r_\xi(1) < 1$. Since $G^r_{\xi}(x)$ is continuous, if $p<p_c(T_\xi,r)$, then $\sup_{x \in (0,1]} G_\xi^r(x) \geq \frac{1}{1-p}$ and if $p_c(T_\xi, r)<p<1$, then for every $x \in (0,1]$, $G_\xi^r(x) < \frac{1}{1-p}$.  Thus, the critical probability for $r$-neighbour bootstrap percolation on the Galton--Watson tree $T_{\xi}$ is, almost surely, given by
\begin{equation}
\label{eq:p_cWithMaxG}
 p_c(T_{\xi}, r) = 1- \frac{1}{\max_{x \in [0,1]} G_{\xi}^r(x)}.
\end{equation}
Since $\max_{x \in [0,1]} G_\xi^r(x) \geq 1$, this implies the following useful estimate for the critical probability
\begin{equation}
\label{eq:p_cWithMaxG2}
 p_c(T_{\xi},r) \leq \max_{x \in [0,1]} G_{\xi}^r(x) -1.
\end{equation}

Before proceeding, a few facts about the functions $g_k^r(x)$ are noted.  First, for all $r \geq 2$,
\begin{equation}
\begin{split}
 \label{eq:grr}
g_r^r(x) & = \frac{\PP(\Bin(r, 1-x) \leq r-1)}{x} = \frac{1-(1-x)^r}{1-(1-x)} \\
& = 1+ (1-x) + (1-x)^2 + \ldots + (1-x)^{r-1} = \sum_{i = 0}^{r-1} (1-x)^i.
\end{split}
\end{equation}
For any $k > r$, $\PP(\Bin(k, 1-x) \leq r) = \PP(\Bin(k , 1-x) \leq r -1)+ \PP(\Bin(k, 1-x)=r)$ and hence
\begin{equation}
\label{E:r-recursion}
g_k^{r+1}(x) = g_k^r(x) + \binom{k}{r} x^{k-r-1}(1-x)^r.
\end{equation}
For each fixed $r \geq 2$ and $k \geq r$,
\begin{equation}\label{E:k-recursion}
g_{k+1}^r(x) - g_k^r(x) = -\binom{k}{r-1}x^{k-r}(1-x)^r.
\end{equation}
Indeed, to prove equation \eqref{E:k-recursion}, let $X \sim \Bin(k, 1-x)$ and $Y \sim \Bin(1, 1-x)$ be independent.  Then, $X+Y \sim \Bin(k+1, 1-x)$ and so
\begin{align*}
 x g_k^r(x) &=\PP(X \leq r-1)\\
	    &=\PP(X+Y \leq r-1) + \PP(Y=1 \text{ and } X=r-1)\\
	    &=x g_{k+1}^r(x) + (1-x) \cdot \binom{k}{r-1} (1-x)^{r-1} x^{k-r+1}\\
	    &=x\left(g_{k+1}^r(x) + \binom{k}{r-1} (1-x)^r x^{k-r} \right),
\end{align*}
which shows equation \eqref{E:k-recursion}.
Thus, by equation \eqref{E:k-recursion}, for any $k \geq r$,
\begin{equation}\label{E:gk-gr}
 g_{k+1}^r(x) = g_r^r(x) - \sum_{i=r}^k \binom{i}{r-1}x^{i-r}(1-x)^r \leq g_r^r(x).
\end{equation}
In particular, note that $G_{\xi}^r(x) \leq g_r^r(x)$.

One simple example of a Galton--Watson tree occurs when the offspring distribution is constant.  When $\xi \equiv b$, $T_\xi$ is the $b$-ary tree, which has the same critical probability as the $(b+1)$-regular tree, $T_b$.  Note that, in this case, $G_\xi^r(x) = g_b^r(x)$.  For $r \geq 2$, fixed, the asymptotic value of $p_c(T_b,r)$ as $b$ tends to infinity is included here for completeness.

\begin{lemma}
\label{lem:pc_regtree}
For each $r \geq 2$, $p_c(T_b,r) = \left(1-1/r\right) \left(\frac{(r-1)!}{b^r} \right)^{1/(r-1)}(1+o(1))$ as $b \to \infty$.
\end{lemma}

\begin{proof}
Fix $r \geq 2$ and $b \geq r$.  The critical probability for $T_b$ in $r$-neighbour bootstrap percolation is given by
\begin{equation}
\label{eq:p_cWithMaxG_r}
p_c(T_b,r) = 1 - \frac{1}{\max_{x \in [0,1]} g_b^r(x)} = \frac{\max_{x \in [0,1]} g_b^r(x) -1}{\max_{x \in [0,1]} g_b^r(x)}.
\end{equation}
For a lower bound on the critical probability, note that 
\[
g_b^r(1-y) = \frac{\PP(\Bin(b,y) \leq r-1)}{1-y} = \frac{1-\PP(\Bin(b,y) \geq r)}{1-y} \geq \frac{1- \binom{b}{r} y^r}{1-y} \geq \frac{1-\frac{(by)^r}{r!}}{1-y}.
\]
Set $y_0 = \left(\frac{(r-1)!}{b^r} \right)^{1/(r-1)}$ so that $b^r y_0^{r-1} = (r-1)!$ and consider
\[
g_b^r(1-y_0) -1\geq \frac{y_0 - \frac{(by_0)^r}{r!}}{1-y_0} = \frac{y_0\left(1-\frac{1}{r} \right)}{1-y_0}. 
\]
Then, a lower bound on the critical probability is given by
\[
p_c(T_b,r) \geq \frac{(1-1/r)\frac{y_0}{1-y_0}}{1+(1-1/r)\frac{y_0}{1-y_0}} = \frac{(1-1/r)y_0}{1-\frac{y_0}{r}} \geq \left(1- \frac{1}{r} \right)y_0 = \left(1 - \frac{1}{r} \right) \left(\frac{(r-1)!}{b^r} \right)^{1/(r-1)}.
\]

For an upper bound of the function $g_b^r(1-y)$, consider separately different ranges for the value of $y$.  Using Chebyshev's inequality, one can show that if $y \geq 2r/b$, then $g_b^r(1-y) \leq 1$.   

Consider the function 
\begin{equation}\label{eq:ax}
 (1-y)(g_b^r(1-y)-1) =  \PP(\Bin(b,y)\leq r-1) -(1-y)=y- \PP(\Bin(b,y) \geq r).
\end{equation}
Suppose that $b > e^{4r}r$  and consider $y$ such that $(r^r
e^{4r}b^{-r})^{1/(r-1)} < y < 2r/b$.  Then  $2r/b < 1/2$ and
\begin{align*}
y-\PP(\Bin(b,y) \geq r) &\leq y- \binom{b}{r} y^r (1-y)^{b-r}\\
	      &\leq y-\frac{b^r}{r^r} y^r e^{-2yb} \leq y-y\frac{b^r y^{r-1}}{r^r} e^{-4r}\\
		&=y\left(1 - y^{r-1} \frac{b^r }{e^{4r} r^r} \right) < 0.
\end{align*}
Consider now 
$y \le \left(\frac{r^r e^{4r}}{b^r} \right)^{1/(r-1)}$. Using equation \eqref{eq:diff_binom_prob} with $y$ in place of $1-x$, the maximum
value for $(1-y)(g_b^r(1-y)-1)$ occurs at $y_1$ with $\PP(\Bin(b-1,
y_1)=r-1) = \frac{1}{b}$ and hence $\binom{b}{r} y_1^{r-1} (1-y_1)^{b-r} =
1/r$.  Thus, 
\begin{equation}
\label{eq:boundOnEnumerator}
 y - \PP(\Bin(b,y) \geq r) \leq y_1- \PP(\Bin(b,y_1) = r) = y_1\left( 1-\frac{1}{r}\right).
\end{equation}
By the choice of $y_1$,
\begin{equation}
\begin{split}
\label{eq:boundOnY_1}
y_1^{r-1}
	& = \frac{1}{b\binom{b-1}{r-1}} (1-y_1)^{-(b-r)} \\ 
	& \leq \frac{(r-1)!}{b^r} \frac{b^r}{b(b-1)\ldots (b-r+1)} e^{2y_1 b} \\
	& = \frac{(r-1)!}{b^r} (1+o(1)).
\end{split}
\end{equation}
Thus, by 
\eqref{eq:ax}, \eqref{eq:boundOnEnumerator} and \eqref{eq:boundOnY_1}, 
\begin{align}
\max_{y\in[0,1]}(g_b^r(1-y)-1)
&\leq \frac{1}{1-\left(\frac{r^r e^{4r}}{b^r} \right)^{1/(r-1)}}\left(1-\frac{1}{r}\right) y_1 \notag\\
     &\leq \left(1-\frac{1}{r}\right) \left(\frac{(r-1)!}{b^r} \right)^{1/(r-1)} (1+o(1)). \label{eq:ub_g-1}
\end{align}
and the upper bound on $p_c(T_b,r)$ follows from \eqref{eq:p_cWithMaxG_r}.
\end{proof}

\subsection{Bounds for $f_r^{GW}\!(b)$}\label{subsec:small_crit_prob}

With the definitions from section \ref{subs:GW_defs}, we are now ready to prove
Theorem \ref{thm:branchingBound}:  For every $r \geq 2$ there are positive constants
$c_r$ and $C_r$ so that for every $b \geq r$, 
\[
 \frac{c_r}{b}e^{-\frac{b}{r-1}} \le f_r^{GW}\!(b) \le C_r e^{-\frac{b}{r-1}}. 
\]
 The proof of Theorem \ref{thm:branchingBound} is given in two parts.  The lower bound for $f_r^{GW}\!(b)$ is given in Lemma \ref{lemma:pcGWlowerBd}, to come, by examining properties of the function $G_{\xi}^r(x)$.  The upper bound for $f_r^{GW}\!(b)$ is given in Lemma \ref{lemma:pcGWupperBd} by producing a family of Galton--Watson trees with fixed branching number and small critical probability for $r$-neighbour bootstrap percolation.

\begin{lemma}
\label{lemma:pcGWlowerBd}
For every $r \geq 2$ and for any offspring distribution $\xi$ with
$\EE(\xi)=b \ge r$, 
\[
p_c(T_{\xi}, r) \geq \frac{e^{-\frac{r-2}{r-1}}}{b} e^{-\frac{b}{r-1}}.
\]
\end{lemma}

\begin{proof}
In what follows, we shall need to consider integrals of functions related to $g_k^r(x)$ and so recall from the definition of the beta function that for all $a , b \in \ZZ^+$,
\[
\int_0^1 x^a (1-x)^b\ dx = \frac{a!\  b!}{(a+b+1)!}.
\]
By equation \eqref{E:gk-gr}, for any $k \geq r$, using $H_\ell = \sum_{i=1}^{\ell} \frac{1}{i}$ to denote the $\ell$-th harmonic number,
\begin{equation}
\label{eq:g_k^rIntegral}
\begin{split}
\int_0^1 \frac{g_r^r(x) - g_{k}^r(x)}{(1-x)^2}\ dx
	&=\sum_{i=r}^{k-1} \binom{i}{r-1} \int_0^1 x^{i-r}(1-x)^{r-2}\ dx\\
	&=\sum_{i=r}^{k-1} \binom{i}{r-1} \frac{(i-r)! (r-2)!}{(i-1)!}\\
	&=\sum_{i=r}^{k-1} \frac{1}{r-1} \frac{i}{i-r+1}\\
	&= \frac{k - r}{r-1} +H_{k -r}.
\end{split}
\end{equation}
Therefore, for any offspring distribution $\xi$, since $\xi \geq r$ almost surely,
\begin{equation}
\label{eq:upperBound}
\begin{split}
\int_0^1 \frac{g_r^r(x) - G_{\xi}^r(x)}{(1-x)^2}\ dx
	&= \sum_{k \geq r}\PP(\xi=k) \left(\frac{k- r}{r-1} +H_{k -r}\right)\\
	 &= \frac{\EE{\xi}}{r-1} + \EE(H_{\xi-r})- \frac{r}{r-1}.
\end{split}
\end{equation}
On the other hand, let $M = \max_{x \in [0,1]} G_{\xi}^r(x)$.  Then by equation \eqref{eq:p_cWithMaxG}, $p_c =
p_c(T_{\xi}, r) = 1-\frac{1}{M}$.  Note that, since $g_r^r(x)$ is decreasing
and continuous, $g_r^r(0)=r$, $g_r^r(1)=1$ and $G_{\xi}^r(x) \leq g_r^r(x)$,
we have $M \in [1,r]$ and there is a unique $y \in [0,1]$ with $g_r^r(1-y) =
M$. Then, by \eqref{eq:grr},
\begin{align*}
\int_0^{1-y} \frac{g_r^r(x) - M}{(1-x)^2}\ dx
	&=\left\{-\frac{M-1}{1-x} - \log(1-x) - \sum_{i=2}^{r-1} \frac{(1-x)^{i-1}}{i-1}\right\}_{x=0}^{1-y}\\
	&= (M-1)(1-1/y) -\log(y) + \sum_{i=1}^{r-2} \frac{1-y^i}{i}.
\end{align*}
Note that $(M-1)(1-1/y) = \frac{(y+y^2 + \ldots + y^{r-1})(y-1)}{y} = y^{r-1}-1$.  Thus, the above expression can be simplified, as
\begin{equation}
\label{eq:integral}
\begin{split}
\int_0^{1-y} \frac{g_r^r(x) - M}{(1-x)^2}\ dx
	& = y^{r-1} -1 - \log(y) + \sum_{i=1}^{r-2} \frac{1-y^i}{i}\\
	& \geq y^{r-1}-1 -\log(y). 
\end{split}
\end{equation}
Now, using the definition of $y$,
\begin{equation}
\label{eq:pc}
p_c = 1 - \frac{1}{M} = \frac{M-1}{M} = \frac{y+y^2+\ldots + y^{r-1}}{1+y+y^2 + \ldots + y^{r-1}} = \frac{y(1-y^{r-1})}{1-y^r}.
\end{equation}
Note that for any $y \in [0,1)$, 
\[
\log \left(\frac{1-y^r}{1-y^{r-1}} \right) \leq \log \left(\frac{1-y^{2r-2}}{1-y^{r-1}} \right) = \log(1+y^{r-1}) \leq y^{r-1}
\]
and from this, using \eqref{eq:pc}, we obtain
\[
y^{r-1}-\log(y) \geq \log\left(\frac{1-y^r}{1-y^{r-1}}\right) - \log(y) =-\log\left(\frac{y(1-y^{r-1})}{1-y^r} \right) = -\log p_c. 
\]
Since $g_r^r(x) - G_{\xi}^r(x) \geq 0$ then, using \eqref{eq:upperBound}
 and \eqref{eq:integral},
\[
-\log p_c -1 \leq \int_0^{1-y}\frac{g_r^r(x) - M}{(1-x)^2}\ dx \leq \int_0^1 \frac{g_r^r(x) - G_{\xi}^r(x)}{(1-x)^2}\ dx = \frac{\EE{\xi}}{r-1} + \EE(H_{\xi-r}) - \frac{r}{r-1}
\]
and hence
\begin{equation}  \label{eq:xb}
p_c(T_{\xi},r) \geq \exp\left(-\frac{\EE(\xi)-1}{r-1}-\EE(H_{\xi-r}) \right)
\geq \exp\left(-\frac{b-1}{r-1}-\EE(H_{\xi}) \right).
\end{equation}
Using the inequality $H_n \leq \log n + 1$ for $n\ge1$
and the concavity of the logarithm function we see that  $\EE(H_\xi)\le \log b+1$ and thus
 \[
p_c(T_{\xi},r) \geq \exp\left(-\frac{r-2}{r-1}\right) \frac{e^{-\frac{b}{r-1}}}{b},
\]
completing the proof of the lemma.
\end{proof}

By Lemma \ref{lemma:pcGWlowerBd}, the  lower bound in Theorem \ref{thm:branchingBound} holds with $c_r = e^{-\frac{r-2}{r-1}}$.

Next let us prove that there exists $C_r>0$ so that $f_r^{GW}\!(b) \leq C_r  e^{\frac{b}{r-1}}$ when $b$ is sufficiently large. We shall do this by first considering a sequence of offspring distributions that are shown to have critical probability $0$.

For each $r \geq 2$, define an offspring distribution, denoted by $\xi_r$ as follows.  For every $k \geq r$, set
\[
\PP(\xi_r = k) = \frac{r-1}{k(k-1)}.
\]
Note that for any $r$, $\EE(\xi_r) = \infty$.  In Lemma \ref{lemma:pcGWupperBd} below, it is shown that, given $b > r$ sufficiently large, the distribution $\xi_r$ can be modified by `pruning' to obtain the appropriate critical probability and mean $b$.
\begin{claim}
\label{claim:infiniteE}
 For each $r \geq 2$, and for all $x \in [0,1]$,  $G_{\xi_r}^r(x) = 1$.
\end{claim}
\begin{proof}
We apply induction on $r$. First, for $r=2$,
\[
\begin{split}
 G^2_{\xi_2}(x) & = \sum_{k \geq 2} \frac{1}{k(k-1)} \left( k x^{k-2}-(k-1) x^{k-1} \right) \\
		& = 1-\sum_{k \geq 2} \left( \frac{1}{k(k-1)} (k-1) x^{k-1} - \frac{1}{(k+1)k} (k+1) x^{(k+1)-2} \right) \\
		& = 1-\sum_{k \geq 2} 0 = 1,
\end{split}
\]
as claimed. Turning to the induction step, assume that the Claim holds for $r \geq 2$: $G_{\xi_r}^r(x) = 1$ for $x \in [0,1)$. Then, for $x \in [0,1)$,
\begin{align*}
 G^{r+1}_{\xi_{r+1}}(x) & = \sum_{k \geq r+1} \frac{r}{k(k-1)} g^{r+1}_k(x)\\
& =\sum_{k \geq r+1} \frac{r}{k(k-1)} \left( g^{r}_k(x) + \binom{k}{r} x^{k-r-1} (1-x)^r \right) \qquad \text{(by \eqref{E:r-recursion})}\\
& =\frac{r}{r-1} \left( \sum_{k \geq r} \frac{r-1}{k(k-1)} g^{r}_k(x) - \frac{1}{r} g^{r}_r(x) \right) + \sum_{k \geq r+1} \frac{1}{r-1} \binom{k-2}{r-2}x^{k-r-1}(1-x )^r\\
& =\frac{r}{r-1} G^{r}_{\xi_{r}}(x) - \frac{1}{r-1} \left( g^{r}_r(x) - \frac{1-x-(1-x)^r}{x} \right) \\
& =\frac{r}{r-1} - \frac{1}{r-1} \left( \frac{1-(1-x)^r}{x} - \frac{1-x-(1-x)^r}{x} \right) \qquad \text{(by \eqref{eq:grr})}  \\
& =\frac{r}{r-1} - \frac{1}{r-1} = 1,
\end{align*}
so our claim holds for $r+1$, completing the proof.
\end{proof}

An immediate corollary of Claim \ref{claim:infiniteE} is that, for every $r \geq 2$, the Galton--Watson tree $T_{\xi_r}$ satisfies $p_c(T_{\xi_r}, r) = 0$.

\begin{lemma}\label{lemma:pcGWupperBd}
For every $r \geq 2$, there is a constant $C_r$ such that if $b \geq 
(r-1)\log(4er)$, then there is an offspring distribution $\eta_{r,b}$ with
$\EE(\eta_{r,b}) = b$ and 
\[
p_c(T_{\eta_{r,b}}, r) \leq C_r e^{-\frac{b}{r-1}}.
\]
\end{lemma}

\begin{proof}
If $b$ is sufficiently large, the distribution $\eta_{r,b}$ is constructed by restricting the support of the distribution $\xi_r$ to a finite set of integers and redistributing the remaining measure suitably. Note that for $m \geq r$ we have
\begin{equation}
\label{eq:probabilityReminder}
 \PP(\xi_r \leq m) = \sum_{k=r}^{m} \PP(\xi_r = k) = (r-1) \sum_{k=r}^{m} \left( \frac{1}{k-1} - \frac{1}{k} \right) = 1 - \frac{r-1}{m}.
\end{equation}
Also, using the convention that $H_0 = 0$,
\[
 \sum_{k=r}^{m} k \PP(\xi_r = k) = (r-1) \sum_{k=r}^{m} \frac{1}{k-1} = (r-1) \left(H_{m-1} - H_{r-2}\right)
\]
is the part of the expected value contributed by the $(m-r+1)$ smallest possible values of $\xi_r$. Given $b$ and $r$, let
\[
 k_0 = \max\{m: (r-1) \left(H_{m-1} - H_{r-2}\right) \leq b\}.
\]
Then, by the choice of $k_0$,
\[
 b < (r-1) \left(H_{k_0} - H_{r-2}\right) < (r-1) H_{k_0} \leq (r-1) (\log k_0 + 1), 
\]
so $k_0 > e^{\frac{b-r+1}{r-1}} \geq 4r$ for 
$b \geq (r-1)\left(\log(4r)+1 \right)=(r-1)\log(4er)$.

Let $k_1 = k_0 - 2r >r$. Then by equation \eqref{eq:probabilityReminder} we have
\[
A = 1-\sum_{k=r}^{k_1} \PP(\xi_r = m) = \frac{r-1}{k_1} = \frac{r-1}{k_0-2r}.
\]
Define $K = b - \sum_{k=r}^{k_1} k \PP(\xi_r = k)$, roughly thought of as the unallocated portion of the expected value.  Then, $K$ can be bounded from below by
\[
K \geq \sum_{k=k_1+1}^{k_0} k \PP(\xi_r = k) = (r-1)\left(H_{k_0 - 1} - H_{k_1-1} \right) \geq (r-1) \frac{2r}{k_0}.
\]
Since $b < \sum_{k=r}^{k_0 +1} k \PP(\xi_r = k)$, we have that
\[
K < \sum_{k=k_1+1}^{k_0+1} k \PP(\xi_r = k) =(r-1)\left(H_{k_0} - H_{k_1 - 1} \right) \leq (r-1) \frac{2r+1}{k_0-2r}.
\]
Thus, it follows that $K/A \leq 2r+1$ and for $k_0 > 4r$,
\[
K/A \geq 2r \left(\frac{r-1}{k_0}\right) \left(\frac{k_0-2r}{r-1}\right) = 2r \left(\frac{k_0-2r}{k_0}\right) > r. 
\]
This implies that, for $b > (r-1)\log(4er)$, there exists $\alpha \in (0,1)$ such that $\frac{K}{A} = \alpha r + (1-\alpha)(2r+1)$ and hence,
\[
 \sum_{k=r}^{k_1} k \PP(\xi_r = k) + \alpha A r + (1-\alpha) A (2r+1) = b.
\]
This is used to define the pruned offspring distribution $\eta_{r, b}$ as follows,
\[
\PP(\eta_{r,b} = k) = 
	\begin{cases}
		\PP(\xi_r = k)	&\text{for } r < k \leq k_1,\ k \neq 2r+1\\
		\PP(\xi_r = r) + \alpha A	&\text{for } k =r, \text{ and}\\
		\PP(\xi_r = 2r+1) + (1-\alpha)A 	&\text{for } k=2r+1.
	\end{cases}
\]
Note that since $k_0 > 4r$, $k_1 = k_0 - 2r > 2r$.

This pruning $\eta_{r,b}$ of the distribution of $\xi_r$ is used to give an upper bound on $f_r^{GW}\!(b)$. Recall that for every $k \geq r$, the functions $g^{r}_k(x)$, given by Definition \ref{def:G_xi},  are non-negative and by equation \eqref{E:gk-gr},  $g^r_k(x) \leq g_r^r(x)$.  By Claim \ref{claim:infiniteE}, $G_{\xi_r}^r(x) = 1$ which shows that, 
\[
G_{\eta_{r,b}}^r(x) \leq G_{\xi_r}^r(x) + \alpha A g_r^r(x) + (1-\alpha)A g_{2r+1}^r(x) \leq 1 + A g_r^r(x).
\]
Therefore, since $g_r^r(x)$ is decreasing
and $g^r_r(0)=r$, 
\[
 \max_{x \in [0,1]} G_{\eta_{r, b}}^r(x) \leq 1 + A g_r^r(0) 
=1+Ar,
\]
and so
\[
p_c(T_{\eta_{r,b}}, r) \leq 
Ar = \frac{r(r-1)}{k_0-2r}
<\frac{r(r-1)}{e^{\frac{b-r+1}{r-1}}-2r} < 2er(r-1) e^{-\frac{b}{r-1}}
\]
for $b > (r-1) \log(4er)$.
\end{proof}

Thus the upper bound in Theorem \ref{thm:branchingBound} holds with $C_r =
2er(r-1)$ for $b \ge (r-1) \log(4er)$, 
and it is trivially true for some $C_r$ for smaller $b$. 
This completes the proof of the theorem.

\subsection{Bounds for $p_c(T_{\xi},r)$}\label{sec:bounds_pc}

\subsubsection{Bounds based on higher moments}
\label{sec:1+alpha}

In this section, we shall prove a lower bound on the critical probability $p_c(T_{\xi}, r)$ based on the $(1+\alpha)$-moments of the offspring distribution $\xi$ for all $\alpha \in (0,1)$, using a modification of the proof of Lemma \ref{lemma:pcGWlowerBd} together with some properties of the gamma function and the beta function.  

Recall that the gamma function is given, for $z$ with $\Re(z) > 0$, by $\Gamma(z) = \int_0^{\infty} t^{z-1} e^{-t}\ dt$ and for all $n \in \mathbb{Z}^+$, satisfies $\Gamma(n) = (n-1)!$.  The beta function is given, for $\Re(x), \Re(y)>0$, by $\B(x,y) = \int_0^1 t^{x-1}(1-t)^{y-1}\ dt$ and satisfies $\B(x,y) = \frac{\Gamma(x)\Gamma(y)}{\Gamma(x+y)}$. We shall use the following bound on the ratio of two values of the gamma function obtained by Gautschi \cite{gammainequalities}. For $n \in \NN$ and $0 \leq s \leq 1$,
\begin{equation}
\label{eq:gammasIneq}
 \left( \frac{1}{n+1} \right)^{1-s} \leq \frac{\Gamma(n+s)}{\Gamma(n+1)} \leq \left( \frac{1}{n} \right)^{1-s}.
\end{equation}
The proof of Theorem \ref{thm:(1+alpha)Bound} if first given for the case $\alpha \in (0,1)$.  For $r \geq 3$ and $\alpha=1$, we then deduce a lower bound for $p_c(T_\xi, r)$ by a
 continuity argument. An analogous bound for $r = 2$ and $\alpha = 1$ is
given in Theorem \ref{thm:2ndMomentBound}. 

\begin{proofOfTheorem}{thm:(1+alpha)Bound}
The proof of the lower bound in this theorem is similar to that of Theorem \ref{thm:branchingBound}, using bounds on integrals similar to the ones in \eqref{eq:g_k^rIntegral} and \eqref{eq:integral}, but with $(1-x)^{2+\alpha}$ in the denominator instead of $(1-x)^2$.

Let $r \geq 2$ and let $0 < \alpha < 1$. From the definition of the beta function, for any $k > r$, we have
\begin{align*}
\int_0^1 \frac{g_r^r(x) - g_{k}^r(x)}{(1-x)^{2+\alpha}}\ dx &=\sum_{i=r}^{k-1} \binom{i}{r-1} \int_0^1 x^{i-r}(1-x)^{r-2-\alpha}\ dx \\
& =\sum_{i=r}^{k-1} \binom{i}{r-1} \B(i-r+1,r-1-\alpha).
\end{align*}
Continuing we obtain
\begin{align*}
\sum_{i=r}^{k-1} \binom{i}{r-1} \B(i-r+1,r-1-\alpha) & = \sum_{i=r}^{k-1} \binom{i}{r-1} \frac{(i-r)!\Gamma(r-1-\alpha)}{\Gamma(i-\alpha)} \\
& = \sum_{i=r}^{k-1} \frac{1}{i-r+1} \frac{i!}{\Gamma(i-\alpha)} \frac{\Gamma(r-1-\alpha)}{(r-1)!}.
\end{align*}
Using inequality \eqref{eq:gammasIneq} we have
\[
 \frac{i!}{\Gamma(i-\alpha)} = i \frac{\Gamma(i)}{\Gamma(i-\alpha)} = i \frac{\Gamma(i-1+1)}{\Gamma(i-1+(1-\alpha))} \leq i^{1+\alpha}.
\]
The further steps depend on the value of $r$. First we consider the case $r \geq 3$. This implies, again using inequality \eqref{eq:gammasIneq},
\[
 \frac{\Gamma(r-1-\alpha)}{(r-1)!} = \frac{1}{r-1} \frac{\Gamma(r-2+(1-\alpha))}{\Gamma(r-2+1)} \leq \frac{1}{(r-1)(r-2)^\alpha} < \frac{1}{(r-2)^{1+\alpha}}.
\]
Thus, putting these together, bounding crudely we find that for $r \geq 3$
\begin{align*}
\int_0^1 \frac{g_r^r(x) - g_{k}^r(x)}{(1-x)^{2+\alpha}}\ dx & < \sum_{i=r}^{k-1} \frac{1}{i-r+1} \left( \frac{i}{r-2} \right)^{1+\alpha} \\
& < \frac{k^\alpha}{(r-2)^{1+\alpha}} \sum_{i=r}^{k-1} \frac{i}{i-r+1} \\
& = \frac{k^\alpha}{(r-2)^{1+\alpha}} \left( k-r+(r-1)H_{k-r} \right) \\
& < \left( \frac{k}{r-2} \right)^{1+\alpha} + 2 \left( \frac{k}{r-2} \right)^{\alpha} H_{k-r} \\
&< \frac{3 k^{1+\alpha}}{(r-2)^\alpha}.
\end{align*}
Now we consider the case $r=2$. We have
\[
 \frac{\Gamma(r-1-\alpha)}{(r-1)!} = \Gamma(1-\alpha) = \frac{\Gamma(2-\alpha)}{1-\alpha} < \frac{1}{1-\alpha}.
\]
Thus a corresponding bound on our integral is
\[
 \int_0^1 \frac{g_2^2(x) - g_{k}^2(x)}{(1-x)^{2+\alpha}}\ dx < \frac{k^{1+\alpha} + k^\alpha H_{k-2}}{1-\alpha} < \frac{2 k^{1+\alpha}}{1-\alpha}.
\]
Thus, proceeding analogously to \eqref{eq:upperBound} we have
\begin{equation}
\label{eq:(1+alpha)UpperBound}
\int_0^1 \frac{g_r^r(x) - G_{\xi}^r(x)}{(1-x)^{2+\alpha}}\ dx < \begin{cases}
\frac{2 \EE(\xi^{1+\alpha})}{1-\alpha},      & \text{if }r=2,\vspace{0.5em}\\
\frac{3 \EE(\xi^{1+\alpha})}{(r-2)^{\alpha}},& \text{otherwise.}
\end{cases}
\end{equation}
Let us now bound our integral from below by some function of $p_c$. Again,
for an offspring distribution $\xi$ let $M = \max_{x \in [0,1]}
G_{\xi}^r(x)$. Let us recall that we have $p_c = p_c(T_{\xi}, r) =
1-\frac{1}{M}$. Recall also that, since $g_r^r(x)$ is decreasing and
continuous, $g_r^r(0)=r$, $g_r^r(1)=1$ and $G_{\xi}^r(x) \leq g_r^r(x)$, we
have $M \in [1,r]$ and there is a unique $y \in [0,1]$ with $g_r^r(1-y) =
M$. Thus $M = 1+y + \ldots + y^{r-1}$ and so  
\begin{equation}
\label{eq:p_cBoundx}
p_c 
= 1-\frac1M
= \frac{y(1-y^{r-1})}{1-y^r} 
\ge
 \frac{r-1}{r} y,
\end{equation}
using ${1-y^r}  \le \frac{r}{r-1} (1-y^{r-1})$.  A lower bound on the integral in question is given by
\begin{align}
\int_0^1 \frac{g_r^r(x) - G_{\xi}^r(x)}{(1-x)^{2+\alpha}}\ dx
	&\ge \int_0^{1-y} \frac{g_r^r(x) - M}{(1-x)^{2+\alpha}}\ dx \notag\\
	&=  \int_0^{1-y} \frac{\sum_{i=0}^{r-1}(1-x)^i - M}{(1-x)^{2+\alpha}}\ dx \notag\\
	&= \Biggl\{- \frac{M-1}{(1+\alpha)(1-x)^{1+\alpha}} + \frac{1}{\alpha(1-x)^{\alpha}} - \sum_{i=1}^{r-2} \frac{(1-x)^{i-\alpha}}{i-\alpha}\Biggr\}_{x=0}^{1-y} \notag\\
	&= -\frac{M-1}{(1+\alpha)y^{1+\alpha}} + \frac{M-1}{1+\alpha} + \frac{1}{\alpha y^{\alpha}} - \frac{1}{\alpha} + \sum_{i=1}^{r-2} \frac{1-y^{i-\alpha}}{i-\alpha} \notag\\
	&\ge -\frac{\sum_{i=0}^{r-2} y^i}{(1+\alpha)y^{\alpha}} + \frac{\sum_{i=1}^{r-1} y^i}{1+\alpha} + \frac{1}{\alpha y^{\alpha}} - \frac{1}{\alpha}. \label{eq:(1+alpha)lb}
\end{align}	
The approximations for the cases $r=2$ and $r\geq 3$ are dealt with separately.  In the case $r=2$, the expression in \eqref{eq:(1+alpha)lb} reduces to
\begin{equation}\label{eq:(1+alpha)lb_r=2}
-\frac{y}{(1+\alpha)y^{1+\alpha}} + \frac{y}{1+\alpha} 
+ \frac{1}{\alpha y^{\alpha}} - \frac{1}{\alpha} 
\ge 
\frac{1}{\alpha(1+\alpha) y^{\alpha}} - \frac{1}{\alpha} .
\end{equation}
For $r \geq 3$, the expression in \eqref{eq:(1+alpha)lb} is bounded from below as follows:
\begin{align}
-\frac{\sum_{i=0}^{r-2} y^i}{(1+\alpha)y^{\alpha}} + \frac{\sum_{i=1}^{r-1} y^i}{1+\alpha} + \frac{1}{\alpha y^{\alpha}} - \frac{1}{\alpha}  \notag
	& \ge -\frac{1 + y}{(1+\alpha)y^{\alpha}} + \frac{y^{r-2+\alpha} + y^{r-1+\alpha}}{(1+\alpha)y^{\alpha}} + \frac{1}{\alpha y^{\alpha}} - \frac{1}{\alpha} \notag\\
	& \ge \frac{- \alpha - \alpha y + 2 \alpha y^{r-1+\alpha} + 1 + \alpha}{\alpha(1+\alpha)y^{\alpha}} - \frac{1}{\alpha} \notag\\
	& = \frac{1 - \alpha (y - 2 y^{r-1+\alpha})}{\alpha(1+\alpha)y^{\alpha}} - \frac{1}{\alpha}. \label{eq:(1+alpha)lb_r>=3}
\end{align}

Define $h_{r,\alpha}(y) = 1 - \alpha (y - 2 y^{r-1+\alpha})$ when $r\ge3$ and 
$h_{r,\alpha}(y)=1$ when $r=2$.  By inequalities \eqref{eq:(1+alpha)lb}, \eqref{eq:(1+alpha)lb_r=2}, and \eqref{eq:(1+alpha)lb_r>=3}, for every $r \geq 2$,
\[
 \int_0^1 \frac{g_r^r(x)-G_\xi^r(x)}{(1-x)^{2+\alpha}}\ dx \geq \frac{h_{r,\alpha}(y)}{\alpha(1+\alpha)y^{\alpha}} - \frac{1}{\alpha}.
\]

For $r \geq 3$, the minimum of $h_{r,\alpha}$ in the interval $[0,1]$ is positive and is
attained at 
\[
 y = b_{r,\alpha} = \left( \frac{1}{2(r+\alpha-1)} \right)^\frac{1}{r+\alpha-2}.
\]
 Thus if $y < c'_{r,\alpha}=
\left( \frac{h_{r,\alpha}(b_{r,\alpha})}{2(1+\alpha)}\right)^{1/\alpha}$ then
\[
\frac{h_{r,\alpha}(y)}{\alpha(1+\alpha)y^{\alpha}} \geq \frac{h_{r,\alpha}(b_{r,\alpha})}{\alpha(1+\alpha)y^{\alpha}} > \frac{2}{\alpha},
\]
and so in this case we obtain
\begin{equation}
\label{eq:lowerBound}
\int_0^1 \frac{g_r^r(x) - G_{\xi}^r(x)}{(1-x)^{2+\alpha}}\ dx 
> \frac{1}{2} \frac{h_{r,\alpha}(b_{r,\alpha})}{\alpha(1+\alpha)y^{\alpha}}
\end{equation}
and thus, combining \eqref{eq:(1+alpha)UpperBound} and \eqref{eq:lowerBound},
$y> c''_{r,\alpha} (\EE(\xi^{1+\alpha}))^{-1/\alpha}$
with
\begin{equation*}
c''_{r,\alpha} = \begin{cases}
\left( \frac{1-\alpha}{4\alpha(1+\alpha)} \right)^{\frac{1}{\alpha}},     & \text{if }r=2,\\[0.5em]
(r-2) \left( \frac{h_{r,\alpha}(b_{r,\alpha})}{6\alpha(1+\alpha)} \right)^{\frac{1}{\alpha}},   & \text{otherwise.}
\end{cases}
\end{equation*}
Note that in the case where $y \geq c'_{r, \alpha}$, then $y \geq c'_{r, \alpha} (\EE(\xi^{1+\alpha}))^{-1/\alpha}$ since $\EE(\xi^{1+\alpha})\ge1$.  Thus, using \eqref{eq:p_cBoundx}, 
the theorem holds for $\alpha \in (0,1)$ with
\[
 c_{r,\alpha} = \frac{r-1}{r}\min(c'_{r,\alpha},c''_{r,\alpha}).
\]
Since for $r \geq 3$ we have $c_{r,\alpha} \to c_{r,1} > 0$ as $\alpha \to 1$, we deduce that Theorem \ref{thm:(1+alpha)Bound} holds for $r \geq 3$ and $\alpha = 1$. However, the value of $c''_{2,\alpha}$ in our proof tends to 0 as $\alpha \to 1$, and consequently so does $c_{2,\alpha}$. We deal with this problem in Theorem \ref{thm:2ndMomentBound} where  an essentially sharp lower bound on $p_c(T_{\xi},2)$ is given based on the second moment of $\xi$, completing also the proof of the lower bound in Theorem \ref{thm:(1+alpha)Bound}.

The upper bound in Theorem \ref{thm:(1+alpha)Bound} follows from Lemma \ref{lem:pc_regtree} and \eqref{eq:ub_g-1} which show that for any $r \geq 2$ there is a constant $A_r > 0$ such that for any 
$k \geq r$, 
\[
 \max_{x \in [0,1]} g_k^r(x) -1 \leq \frac{A_r}{k^{r/(r-1)}}.
\]
Thus the upper bound follows immediately from inequality \eqref{eq:p_cWithMaxG2}.
\end{proofOfTheorem}

\subsubsection{Bounds for $p_c(T_{\xi},2)$}
\label{sec:lb_2nbr}

In this section we focus on $2$-neighbour bootstrap percolation on Galton--Watson trees. This specific problem is easier to tackle analytically which gives us an opportunity to obtain sharp bounds on $p_c(T_{\xi},2)$.  To simplify notation, we write $G_\xi$ for $G^2_\xi$.
\begin{proofOfTheorem}{thm:2ndMomentBound}
First we prove the rather easy bound given in \eqref{eq:dominationBound}. By the definition of function $G_\xi(x)$ we see that for each $k \geq 2$ we have
\[
 G_\xi(x) \geq \PP(\xi=k) g_k^2(x) = \PP(\xi=k) \left( k x^{k-2}-(k-1)x^{k-1} \right).
\]
Now, $g_2^2(x)=2-x$ so it attains its maximum in the interval $[0,1]$ at $x=0$ with $g_2^2(0)=2$, while for $k \geq 3$ functions $g_k^2(x)$ are maximized at $x_k = \frac{k(k-2)}{(k-1)^2}$, with $g_k^2(x_k) = \frac{k^{k-1}(k-2)^{k-2}}{(k-1)^{2k-3}}$. Thus formula \eqref{eq:dominationBound} follows immediately from \eqref{eq:p_cWithMaxG}.  

Considering the maximum value of the function $g^2_k(x)$,
\[
\frac{k^{k-1}(k-2)^{k-2}}{(k-1)^{2k-3}} = \left(\frac{k(k-2)}{(k-1)^2} \right)^{k-1}\left(\frac{k-1}{k-2}\right) = \left(1 - \frac{1}{(k-1)^2}\right)^{k-1}\left(\frac{k-1}{k-2}\right).
\]
One can show, by induction on $t$, that for $k \geq 3$ and $t \geq 1$, 
\[
\left(1 - \frac{1}{(k-1)^2}\right)^{t} \leq 1 - \frac{t}{(k-1)^2}+\frac{t(t-1)}{2(k-1)^4}.
\]
In particular, setting $t=k-1$ in this inequality yields
\[
 \left(1-\frac{1}{(k-1)^2}\right)^{k-1} \leq 1 - \frac{1}{(k-1)}+\frac{(k-2)}{2(k-1)^3} = \frac{(k-2)}{(k-1)}\left(1+\frac{1}{2(k-1)^2} \right)
\]
and hence for $k \geq 3$, and all $x \in [0,1]$, $g^2_k(x) \leq 1+ \frac{1}{2(k-1)^2}$.  The maximum value for $g^2_2(x)$ is $g^2_2(0)=2> 1+ \frac{1}{2}$, but it is certainly true that for all $k \geq 2$, $g^2_k(x) \leq 1 + \frac{1}{2(k-1)^2 - (k-1)} = 1+ \frac{1}{(k-1)(2k-3)}$.  Hence 
\[
G_{\xi}(x) \leq  1+\EE\left(\frac{1}{(\xi-1)(2\xi-3)} \right)
\]
which yields the upper bound given by inequality \eqref{eq:dominationUpperbd}. Note that the first bound in inequality \eqref{eq:dominationUpperbd} is essentially sharp as demonstrated by the $(b+1)$-regular tree.

Now let us prove bound \eqref{eq:2ndMomentBound}. To simplify notation, for every $k$, let $(\xi)_k = \xi(\xi-1)(\xi-2) \ldots (\xi-k+1)$ denote the $k$-th falling factorial.  The goal is to approximate $G_{\xi}(x)$ by a polynomial of degree $2$ whose maximum value can be easily calculated.

Consider the Taylor series for $G_{\xi}(x)$ about $x=1$.  For this, note
that $G_\xi(1) = \sum_{k \geq 2} \PP(\xi=k) = 1$, 
$G_\xi'(1) = \sum_{k \geq  2}\PP(\xi=k)(-1) = -1$ and 
\[
G_\xi''(1) = \sum_{k \geq 2} \PP(\xi=k)(-(k-2)(k+1)) = \sum_{k \geq 2} \PP(\xi=k)(-k(k-1) +2) = -\EE((\xi)_2) + 2.
\]

Note that for all $m \geq 1$, $G^{(m)}_{\xi}(1) < 0$, where it exists.

Set $P_2(x) = 1- (x-1) - \frac{(\EE(\xi)_2 - 2)}{2}(1-x)^2 = 2-x- \frac{(\EE(\xi)_2 - 2)}{2}(1-x)^2$.  It is shown below that for all $x \in [0,1]$, $P_2(x) \leq G_{\xi}(x)$.  Note that
\[
P_2(x) = \sum_{k \geq 2} \PP(\xi=k) \left(g^2_2(x) - \frac{(k^2-k-2)}{2}(1-x)^2 \right).
\]

Recall that, by equation \eqref{E:k-recursion}, for all $x$, 
$g^2_{k+1}(x) - g^2_k(x) = -k x^{k-2}(1-x)^2$.  Thus,
\begin{equation}\label{eq:ny}
  \begin{split}
g^2_{k+1}(x) & + \frac{\left((k+1)^2 - (k+1) - 2\right)}{2}(1-x)^2 - \left(g^2_k(x) + \frac{(k^2 - k -2)}{2}(1-x)^2 \right)\\
		&= -k x^{k-2}(1-x)^2 + \left(\frac{2k-2+2}{2} \right)(1-x)^2\\
		&=k(1-x)^2(1-x^{k-2}).	
  \end{split}
\end{equation}

Considering $G_{\xi}(x) - P_2(x)$, note that for $k = 2$, 
$g^2_k(x) - g^2_2(x) + \frac{(k^2 - k -2)}{2}(1-x)^2 = 0$.  For $k \geq3$,
 by \eqref{eq:ny},
\[
g^2_k(x) - g^2_2(x) + \frac{(k^2-k-2)}{2}(1-x)^2 = \sum_{i=2}^{k-1} i (1-x)^2(1-x^{i-2}) \ge0.
\]
Hence,
\[
G_{\xi}(x) - P_2(x) = \sum_{k \geq 2} \PP(\xi = k) \left(g^2_k(x) - g^2_2(x) + \frac{(k^2 - k -2)}{2}(1-x)^2 \right) \geq 0
\]
and so for all $x$, $G_{\xi}(x) \geq P_2(x)$.

Now, $P_2(x)$ is a parabola which attains its maximum value at $x = 1 - \frac{1}{\EE (\xi)_2 - 2}$ with
\[
P_2\left(1 - \frac{1}{\EE (\xi)_2 - 2} \right) = 1 + \frac{1}{\EE (\xi)_2 - 2} - \frac{1}{2}(\EE (\xi)_2 - 2) \frac{1}{(\EE (\xi)_2 - 2)^2} =1 + \frac{1}{2(\EE (\xi)_2 - 2)}.
\]

This immediately implies a lower bound for the critical probability for $T_{\xi}$,
\[
p_c(T_{\xi}, 2) \geq 1 - \frac{1}{1+ \frac{1}{2\EE (\xi)_2 - 4}} = 1 - \frac{2\EE (\xi)_2 - 4}{2\EE (\xi)_2 - 3} = \frac{1}{2\EE (\xi)_2 - 3}.
\]
\end{proofOfTheorem}

\subsection{Examples}\label{sec:examples}

The $(b+1)$-regular tree shows that one cannot hope for a stronger bound
based on the second moment of $\xi$ than the one given by inequality
\eqref{eq:2ndMomentBound}. What is more, this bound turns out to be an
accurate estimate of critical probability in a number of natural offspring
distributions. A few such examples are examined here for comparison. 
 For simplicity, we consider only $r=2$, and we continue to write
$G_\xi$ for $G^2_\xi$.
In what
follows, the notation $o_b(1)$ is used to denote a function tending to $0$
as $b \to \infty$. 

\subsubsection{$2$ or $a$ children}

For $a \in \NN$ and $b$ with $a \ge b> 2$, consider trees denoted $T_{\xi_{b,a}}$ with offspring distribution $\PP(\xi_{b,a}=2) = \frac{a-b}{a-2}$ and $\PP(\xi_{b,a}=a) = \frac{b-2}{a-2}$. Note that the branching number of $T_{\xi_{b,a}}$ is $\br(T_{\xi_{b,a}}) = \EE(\xi_{b,a}) = b$. We do not present a complete proof of the following theorem. However, sharp lower bounds on $p_c(T_{\xi_{b,a}},2)$ follow from Theorem~\ref{thm:2ndMomentBound}.
\begin{theorem}
\label{thm:GW_2a}
 The critical probability in 2-neighbour bootstrap percolation on $T_{\xi_{b,a}}$ is
\[
p_c(T_{\xi_{b,a}},2) = \max \left \{1-\frac{a-2}{2(a-b)}, \frac{1+o_b(1)}{2ab} \right \},
\]
with the first quantity being always greater for $a \geq 2b-1$ and the second for $a \leq 2b-2$.
\end{theorem}
\qed

The random variable $\xi_{b,a}$ is supported on only two values and so clearly $\EE((\xi_{b,a})_2)$ is finite and the assumptions of Theorem
\ref{thm:2ndMomentBound} are satisfied. We have 
\[
\begin{split}
 \EE((\xi_{b,a})_2) & = \PP(\xi_{b,a}=a)a(a-1) + \PP(\xi_{b,a}=2)2 \\
  		    & = \frac{(b-2)a(a-1)+2(a-b)}{a-2} \\
  		    & < \frac{(b-2)a(a-1)}{a-2} + 2.
\end{split}
\]
Thus, inequality \eqref{eq:2ndMomentBound} yields a lower bound on the critical probability given by 
\[
p_c(T_{\xi_{b,a}},2) > \frac{1}{2\left(\frac{(b-2)a(a-1)}{a-2} + 2\right) -3} = \frac{1}{2\frac{(b-2)a(a-1)}{a-2} + 1} = \frac{1 + o_b(1)}{2ab},
\]
agreeing asymptotically with the correct value for $a \leq 2b-2$.

For $a \geq 2b-1$ we have in fact $p_c(T_{\xi_{b,a}},2) = 1 - \frac{1}{2\PP(\xi_{b,a}=2)}$. The value of critical probability, in this case, tells us what prevents $T_{\xi_{b,a}}$ from percolating when we have $p < p_c(T_{\xi_{b,a}},2)$. Since $\frac{a-b}{a-2} > \frac{1}{2}$, after deleting all vertices of degree $a+1$, the tree almost surely contains infinite components, with all vertices having degree at most 3, with branching number $c = 2\frac{a-b}{a-2} > 1$. Every initially healthy doubly infinite path contained in such subtree is an infinite healthy $1$-fort in $T_{\xi_{b,a}}$. The critical probability for such paths to occur is $1/c$ and so if $1-p > 1/c$ then $T_{\xi_{b,a}}$ almost surely does not percolate. Note that exactly the same arguments can be used to prove the first lower bound in inequality \eqref{eq:dominationBound}.

\subsubsection{Shifted Poisson}

A natural offspring distribution for a Galton--Watson tree is a Poisson
distribution. Since any distribution $\xi$ with $\PP(\xi \leq 1) >0$ has
critical probability $1$, consider a Poisson distribution shifted by
$2$. That is, for each  $b > 2$, 
let $\xi^b_{Po}$ be the
offspring distribution with the property that, for each $k \geq 2$, 
\[
\PP(\xi^b_{Po}=k) = e^{-(b-2)} \frac{(b-2)^{k-2}}{(k-2)!}.
\]
Then, $\EE(\xi^b_{Po}) = b$ and the function $G_{\xi^b_{Po}}(x)$ is given by
\[
\begin{split}
 G_{\xi^b_{Po}}(x) & = \sum_{k\geq 2}e^{-(b-2)} \frac{(b-2)^{k-2}}{(k-2)!}(k x^{k-2}-(k-1)x^{k-1}) \\
		   & = e^{-(b-2)(1-x)}(2+(b-3)x - (b-2)x^2).
\end{split}
\]
Here, the critical probability can be given precisely since the function
$G_{\xi^b_{Po}}$ attains its (global) maximum value when $x= \frac{b-5 +
  \sqrt{(b+3)(b-1)}}{2(b-2)}$, which belongs to $[0,1]$ when $b\ge 7/3$; the
maximum value is 
\[
\exp\left(-\frac{1}{2}(b+1 - \sqrt{(b+3)(b-1)})\right)\left(\frac{-2+\sqrt{(b+3)(b-1)}}{b-2}\right).
\]
Thus, with a little bit of calculation, one can show that, for $b\ge7/3$,
\[
p_c(T_{\xi^b_{Po}}, 2) = 1-\left(\frac{(b-2)e^{\frac{b+1-\sqrt{(b+3)(b-1)}}{2}}}{-2+\sqrt{(b+3)(b-1)}}\right) = \frac{1}{2b^2} + \frac{1}{3b^3}+O\left(\frac{1}{b^4}\right).
\]

One can apply Theorem \ref{thm:2ndMomentBound} to the distribution
$\xi^b_{Po}$ since $\EE((\xi^b_{Po})_2) = b^2-2$. Thus,
 \eqref{eq:2ndMomentBound} yields
\[
p_c(T_{\xi^b_{Po}}, 2) \geq  \frac{1}{2b^2-7} = \frac{1+o_b(1)}{2 b^2}
\]
which is asymptotically correct.

\subsubsection{Shifted geometric distribution}

Consider now a shifted geometric distribution.  
For $b >2$,  let $\xi^b_{g}$ be defined by 
\[
\PP(\xi^b_g=k+2) = \frac{1}{b-1}\left(\frac{b-2}{b-1} \right)^k,
\qquad k \geq 0.
\]
Then, $\EE(\xi^b_{g}) = b$ and the function  $G_{\xi^b_g}$ is given by 
\[
G_{\xi^b_g}(x) = \frac{2(b-1) - (2b-3)x}{((b-1)-(b-2)x)^2},
\]
and attains its maximum when $x = \frac{(2b-5)(b-1)}{(b-2)(2b-3)}$ with
value $\frac{(2b-3)^2}{4(b-1)(b-2)}$.  Thus,
if $b\ge 5/2$, 
\[
p_c(T_{\xi^b_g}, 2) = 1- \frac{4(b-1)(b-2)}{(2b-3)^2} = \frac{1}{(2b-3)^2}.
\]
On the other hand we see that $\EE((\xi^b_{g})_2)= 2(b-1)^2$;
 thus \eqref{eq:2ndMomentBound} yields
\[
p_c(T_{\xi^b_g}, 2) \geq \frac{1}{4(b-1)^2-3} = \frac{1+o_b(1)}{4b^2},
\]
again agreeing asymptotically with the true value.

\section{Final remarks and open problems}
\label{sec:openPr}

In this paper we study general infinite trees and show that for any $b \geq
r$ and any $\epsilon > 0$ there exists a tree with bounded degree, branching
number $\br(T) = b$ and critical probability $p_c(T,r) < \epsilon$. We then
show that, 
by equation \eqref{eq:xb}, 
given an offspring distribution $\xi$ with $\PP (\xi < r) = 0$,
for a Galton--Watson tree $T_\xi$ we almost surely have  
\[
 p_c(T_{\xi},r) \geq \exp\left(-\frac{\EE(\xi)-1}{r-1}-\EE(H_{\xi-r}) \right).
\]
Using the concavity of the logarithm function and, setting $\br(T_\xi) =
\EE(\xi) = b$, this bound was
simplified to $p_c(T_{\xi},r) \geq c_r
\frac{e^{-\frac{b}{r-1}}}{b}$,
as stated in Theorem \ref{thm:branchingBound}. 

However, the bound $\EE(H_{\xi-r}) \leq \log b$ is very weak unless the 
distribution $\xi$ is strongly concentrated around its mean. When $\xi$ is concentrated
though, we already know that $p_c(T_{\xi},r)$ is large, e.g., by Theorems
\ref{thm:(1+alpha)Bound} and \ref{thm:2ndMomentBound}, as well as by the
results for regular trees in \cite{infiniteTrees} and
\cite{bootstrapbethe}. With this in mind we conjecture that the family of
offspring distributions $\eta_{r,b}$ constructed in the proof of Lemma
\ref{lemma:pcGWupperBd} minimizes $p_c(T_{\xi},r)$ up to a factor depending on
$r$ only. 
\begin{conjecture}
 The upper bound in Theorem \ref{thm:branchingBound} is essentially sharp, i.e., for $r \geq 2$ there are constants $c_r$ and $C_r$ such that if $b \geq r$ then
       \[
         c_r e^{-\frac{b}{r-1}} \le f_r^{GW}\!(b) \le C_r e^{-\frac{b}{r-1}}.
       \]
\end{conjecture}

The second conjecture we state in this paper is an extension of Theorem \ref{thm:(1+alpha)Bound} which says that for $\alpha \in (0,1]$ we have $p_c(T_\xi,r) \geq c_{r,\alpha} \left( \EE(\xi^{1+\alpha}) \right)^{-1/\alpha}$. For $r=2$ and $\alpha > 1$ such bound does not hold as is seen by taking $\xi = b$ constant, i.e., a regular tree $T_b$, when $p_c(T_b,2) \sim \frac{1}{2b^2}$. However, turning to Lemma \ref{lem:pc_regtree} we observe that $p_c(T_b,r) \sim c_r b^{-\frac{r}{r-1}}$. This motivates the following conjecture, extending Theorem \ref{thm:(1+alpha)Bound} to all values of $\alpha \leq r-1$.
\begin{conjecture}
 For each $r \geq 2$ and $\alpha \in (0,r-1]$ there exists a constant $c_{r,\alpha}>0$ such that for any offspring distribution $\xi$ we have 
\[
 p_c(T_\xi,r) \ge c_{r,\alpha} \left( \EE(\xi^{1+\alpha}) \right)^{-1/\alpha} .
\]
\end{conjecture}

In Theorems \ref{thm:(1+alpha)Bound} and \ref{thm:2ndMomentBound}, we give upper bounds on $p_c(T_{\xi},r)$ based on the $\left ( \frac{r}{r-1} \right )$-th negative moments of $\xi$. However, the example of the $\xi_{b,a}$ offspring distribution in Theorem \ref{thm:GW_2a} immediately shows that negative moments are not enough to tightly bound the critical probability from above. This motivates the following question.
\begin{question}
 What other characteristics of the distribution $\xi$ lead to upper bounds on $p_c(T_\xi, r)$?
\end{question}

\bibliographystyle{amsplain}

 \bibliography{mylargebib}
 
\end{document}